\documentclass[11pt]{article}
\usepackage{latexsym,amsfonts,amssymb,amsmath,amsthm}
\usepackage{graphicx}
\usepackage{url}
\usepackage[usenames,dvipsnames]{color}
\usepackage{ulem,comment}
\usepackage{tcolorbox}
\usepackage[framed]{matlab-prettifier}
\usepackage{relsize,exscale}
\usepackage{bigints}
\usepackage{caption}
\usepackage{subcaption}
\usepackage{mathrsfs}

\newcommand{\redsout}{\bgroup\markoverwith{\textcolor{red}{\rule[0.5ex]{2pt}{0.4pt}}}\ULon}

\usepackage{color}
\usepackage{array}

\usepackage{float}
\usepackage{lipsum} 
\usepackage{fancyhdr}
\DeclareMathOperator*{\argmin}{argmin} 
\DeclareMathOperator*{\diag}{diag} 
\DeclareMathOperator*{\rank}{rank} 
\DeclareMathOperator*{\cov}{Cov}
\DeclareMathOperator*{\var}{Var}
\DeclareMathOperator*{\bias}{Bias}
\DeclareMathOperator*{\tr}{tr}

\begin{document}






\newtheorem{theorem}{Theorem}[section]
\newtheorem{proposition}[theorem]{Proposition}
\newtheorem{definition}{Definition}[section] 
\newtheorem{example}{Example}[section]
\newtheorem{exercise}{Exercise}[section]
\newtheorem{application}{Application}[section]
\newtheorem{homework}{Homework}
\newtheorem{practice}{Practice}
\newtheorem{corollary}{Corollary}[section]
\newtheorem{lemma}[theorem]{Lemma}
\newtheorem{assumption}[theorem]{Assumption}
\newtheorem{remark}{Remark}[section]
\newtheorem{notation}[theorem]{Notation}
\numberwithin{equation}{section}

\newcommand{\abs}[1]{\left\lvert #1 \right\rvert}
\newcommand{\sn}[2]{\left[ #1 \right]_{#2}}

\newcommand{\stk}[2]{\stackrel{#1}{#2}}
\newcommand{\dwn}[1]{{\scriptstyle #1}\downarrow}
\newcommand{\upa}[1]{{\scriptstyle #1}\uparrow}
\newcommand{\nea}[1]{{\scriptstyle #1}\nearrow}
\newcommand{\sea}[1]{\searrow {\scriptstyle #1}}
\newcommand{\csti}[3]{(#1+1) (#2)^{1/ (#1+1)} (#1)^{- #1
 / (#1+1)} (#3)^{ #1 / (#1 +1)}}
\newcommand{\RR}[1]{\mathbb{#1}}

\newcommand{\rd}{{\mathbb R^d}}
\newcommand{\ep}{\varepsilon}
\newcommand{\rr}{{\mathbb R}}
\newcommand{\alert}[1]{\fbox{#1}}
\newcommand{\eqd}{\sim}
\def\p{\partial}
\def\R{{\mathbb R}}
\def\N{{\mathbb N}}
\def\Q{{\mathbb Q}}
\def\C{{\mathbb C}}
\def\l{{\langle}}
\def\r{\rangle}
\def\t{\tau}
\def\k{\kappa}
\def\a{\alpha}
\def\la{\lambda}
\def\De{\Delta}
\def\de{\delta}
\def\ga{\gamma}
\def\Ga{\Gamma}
\def\ep{\varepsilon}
\def\eps{\varepsilon}
\def\si{\sigma}
\def\Re {{\rm Re}\,}
\def\Im {{\rm Im}\,}
\def\E{{\mathbb E}}
\def\P{{\mathbb P}}
\def\Z{{\mathbb Z}}
\def\D{{\mathbb D}}
\newcommand{\ceil}[1]{\lceil{#1}\rceil}

\title{{Well-posedness and numerical reconstruction of a source term for linear parabolic problems with an integral constraint}}

\author{
  Jason R. Morris\footnote{Department of Mathematics, State University of New York at Brockport, NY 14420 (jrmorris@brockport.edu)} \;and 
  Sedar Ngoma\footnote{Department of Mathematics, State University of New York at Geneseo, NY 14454 (ngoma@geneseo.edu)}
  }
\date{} 
\maketitle



\begin{abstract}
This work investigates a time-dependent source identification problem for
linear parabolic equations subject to an integral constraint and Neumann
boundary conditions in a domain of $\mathbb{R}^d$, $d\ge 1$. We establish
well-posedness and higher regularity of the solution pair in parabolic
H\"older spaces. A numerical algorithm based on a finite element
discretization in space and an implicit time-stepping scheme is then
developed for the reconstruction of the unknown source.

The resulting discrete inverse problem involves a well-conditioned
operator. We show that identity Tikhonov regularization provides only
uniform, nonselective shrinkage in this setting, whereas Tikhonov
regularization with derivative-based penalties, analyzed through the
generalized singular value decomposition, provides an effective
denoising strategy. The regularization parameter is selected using the
Morozov discrepancy principle. Numerical errors are evaluated using full
parabolic H\"older norms, which provide a more comprehensive assessment
of the reconstruction by incorporating errors in the solution, its
derivatives, and the associated H\"older seminorms. Numerical experiments
for smooth and piecewise constant sources demonstrate accurate and robust
reconstructions under increasing levels of noise.
\end{abstract}

\section{Introduction}
Let $\Omega$ be a domain in $\mathbb{R}^{d}$, $d\geq 1$ with boundary $\partial\Omega$ and let $T > 0$ be given. Denote by $Q_T = (0,T)\times\Omega$ the parabolic domain with lateral surface $S_T = (0,T)\times\partial\Omega$. Consider the parabolic problem
\begin{alignat}{2}
\partial_{t}{u}(t,x) + \mathscr{L}u(t,x) &= \phi(t,x), &&\quad (t,x) \in Q_T \label{parab1} \\
u(0,x) &= u_{0}(x), &&\quad x \in\Omega \label{icparab1} \\
\frac{\partial u(t,x)}{\partial \nu} &= g(t,x), &&\quad (t,x) \in S_T \label{bcparab1},
\end{alignat}
where $\nu$ is the outward unit normal to $\partial\Omega$ and $\mathscr{L}$ is the parabolic operator
\begin{equation}\label{operator1}
\mathscr{L}u(t,x) = -\sum_{i,j = 1}^da^{ij}(t,x)\partial_{x_i\,x_j}u(t,x) + \sum_{i=1}^db^i(t,x)\partial_{x_i} u(t,x) + c(t,x)u(t,x).
\end{equation}
Given the coefficients of $\mathscr{L}$, the source function $\phi(t,x)$, the initial condition $u_0(x)$, and the boundary function $g(t,x)$, problem~\eqref{parab1}--\eqref{bcparab1} is referred to as a direct or forward problem. It is widely used in various applications in science and technology, mathematics, physics, engineering, and many other fields. It describes the evolution in time of the density of some quantity $u(t,x)$ such as concentration of a chemical, heat transfer, etc, within the region $\Omega$. The first term on the right-hand side of Equation~\eqref{operator1} describes diffusion, the second quantity describes transport, and the third describes creation or depletion~\cite[p. 372-373]{Evans10}. On the other hand, given the coefficients of $\mathscr{L}$ together with $u_0$ and $g$, the problem of determining the pair $(u,\phi)$ in~\eqref{parab1}--\eqref{bcparab1} is called an inverse source problem. As such, the problem is under-determined and thus an additional condition on the solution $u$ of the direct problem is required in order to recover the source $\phi$. Several different types of additional conditions have been considered in the literature depending on whether the unknown function is either space-dependent or time-dependent or both. These additional conditions include but are not limited to the knowledge of the solution of the direct problem at the final time (final time measurement), knowledge of the integral of the solution of the direct problem, the availability of the solution of the direct problem at an interior spatial domain point, and the knowledge of the normal derivative of the direct problem at a boundary point. 

We consider a time-dependent inverse source problem corresponding to the equations
\begin{alignat}{2}
\partial_{t}{u}(t,x) + \mathscr{L}u(t,x) &= f(t,x) + \phi(t,x), &&\quad (t,x) \in Q_T \label{inv1} \\
u(0,x) &= u_{0}(x), &&\quad x \in\Omega \label{icinv1} \\
\frac{\partial u(t,x)}{\partial \nu} &= g(t,x), &&\quad (t,x) \in S_T \label{bcinv1}\\
\int_{\Omega}u(t,x)\,dx &= \mu(t). &&\quad t \in (0,T] \label{integparab1}
\end{alignat}

The inverse problem is as follows: given the component $\phi(t,x)$ of the source, the initial condition data $u_0(x)$, the boundary function $g(t,x)$, the integral constraint data $\mu(t)$, and the coefficients of the operator $\mathscr{L}$, determine a pair $(u,f)$ satisfying Equations~\eqref{inv1}--\eqref{integparab1}. So far, the problem formulation is underdetermined.  As we discuss this general form of the problem, we will refine the form of \(f(t,x)\) which in turn will lead to a well-posed problem.

The right-hand side of equation~\eqref{inv1} is a volume energy source term that may be due to electric heating, a nuclear source, or frictional heating in heat conduction. It measures the physical effect of an external heat source and thus represents the instantaneous temperature change due to an external heat source. It is also considered as the source pollutant in air pollution modeling. As an illustration, consider the case where a fertiliser from a field is carried by rain into a stream in a form of runoff. This in turn affects aquatic life. Thus, the estimation of this unknown inhomogeneous source term plays an important role in the migration of ground water and environment protection in cities with large populations~\cite{hazanee2013reconstruction}. 

The inverse source problem in the general case where the source depends on both the time and space variables has been considered in the literature but has the disadvantage of the lack of uniqueness results~\cite{Isakov2006}. This issue has led many researchers to consider sources as a linear combination of point sources, or as additive, or multiplicative in various forms~\cite{hasanov2013analysis,hazanee2013reconstruction,hazanee2016reconstruction}.  In particular, there are multiple practical applications for the multiplicative case $f(t,x)=h(t)w(t,x)$, where $w$ is known and $h$ is an unknown time-dependent scaling factor. These types of source functions arise in several physical applications such as heat processes involving radioisotope decay. They also play a role of a control term for the heat equation. 

The initial and boundary conditions are prescribed to ensure the well-posedness of the forward problem for \(u\) in~\eqref{parab1}--\eqref{bcparab1}, given \(\phi\) and the coefficients of $\mathscr{L}$.  The Neumann boundary conditions state that there is a given flux at the surface of the body under study, represented in this case by the boundary $\partial{\Omega}$. If $g=0$, the homogeneous Neumann boundary conditions express that the boundary $\partial{\Omega}$ is insulated. 

Meanwhile, the integral data~\eqref{integparab1}, perhaps available from measurement, makes it possible to recover a source term adjustment \(f(t,x)\) and the corresponding solution $u(t,x)$ without prior knowledge of either of them.
When \(f\) is treated as an unknown, there are many solution pairs $(u,f)$ to the problem~\eqref{inv1}--\eqref{bcinv1}.  Therefore, additional constraints are expected.  To obtain a unique solution pair, consider the effect of the integral constraint~\eqref{integparab1}. If $u(t,x)$ represents the temperature, then $\mu(t)$ can be interpreted as a measurement of thermal energy throughout \(\Omega\) for each time \(t\). It is reasonable to expect a suitable source adjustment $f(t,x)$ to produce the heat measurement $\mu$. We note that the use of additional data~\eqref{integparab1} is encountered in heat transfer applications~\cite{cannon1986diffusion}. It is also interpreted as the bulk age measurement of a sample in geochronology~\cite{GlotovHamesMeirNgoma2}.

Even with the integral constraint~\eqref{integparab1}, the problem still admits many solution pairs $(u,f)$.  Therefore, it is reasonable to seek solutions of particular forms. In this paper, our approach is to construct solutions of the form $u(t,x)=v(t,x)+h(t)w(t,x)$, for some given function $w(t,x)$, where \(v\) is the solution of the forward problem~\eqref{parab1}--\eqref{bcparab1}. This choice corresponds to using \[f(t,x)=h'(t)w(t,x)+h(t)\partial_t w(t,x)+h(t)\mathscr{L}w(t,x)\] and introducing an additional boundary term $h(t) \frac{\partial w(t,x)}{\partial \nu}$. Since $w$ is given, we regard \(u\) and \(h\) as the unknowns in the inverse problem. These statements are formalized in Theorem~\ref{existence-uniqueness}, in which we prove the existence of a unique solution pair $(u,h)$ in the setting of spaces of H\"older continuous functions. Moreover, though 
the terms $f$ and $g$ in~\eqref{inv1} and~\eqref{bcinv1} have been replaced by more complicated expressions, we explain in
Remark~\ref{simpler-forms} how the simpler expressions may be recovered, depending on features of the function $w$.

Related results appear in the literature.  For example,  
Hazanee, Ismailov, Lesnic, and Kerimov~\cite{HazaneeEtAl13} considered a time-dependent
inverse source problem for the heat conduction equation with the integral constraint
\eqref{integparab1} and nonlocal boundary conditions in one space dimension. The model is
of the form
\[
\partial_t u = \partial_{xx}u + r(t)f(x,t),
\]
where the function $f$ is known and the scalar function $r$ is to be determined from the
thermal energy measurement $\mu(t)$. The authors proved existence, uniqueness, and
continuous dependence of solutions using spectral methods based on separation of
variables. Moreover, they provided numerical solutions using the boundary element method
combined with Tikhonov regularization in order to obtain stable reconstructions.

Ginder~\cite{Ginder2010} investigated the model
\[
\partial_t u - \Delta u + h(x,t) = \lambda(t,u),
\]
subject to Dirichlet boundary conditions and the data $\mu(t)$ in a bounded domain of
$\mathbb{R}^d$, $d\ge 1$. The author proved existence and regularity of the solution pair
$(u,\lambda)$ using the discrete Morse flow, an approach also known as a combination of
Rothe’s time-discretization method and the direct method of the calculus of variations.
Glotov, Hames, Meir, and Ngoma~\cite{GlotovHamesMeirNgoma1} studied the model
\[
\partial_t u - \Delta u = f(t),
\]
subject to Dirichlet boundary conditions and the integral constraint $\mu(t)$ in a bounded
domain of $\mathbb{R}^d$, $d\ge 1$. They proved existence and uniqueness of the solution
pair $(u,f)$ using Rothe’s method together with an energy argument, and provided numerical
results based on a finite element discretization in space and a backward Euler method in
time.

A related inverse source problem involving internal measurements and Dirichlet boundary
conditions in one space dimension was considered by Yang, Dehghan, Yu, and
Luo~\cite{YangEtAl2011}. The authors employed Green’s function techniques and the
contraction mapping principle to establish existence and uniqueness of solutions, and then
applied the Landweber iteration method to compute numerical reconstructions. Other inverse
source problems with various types of additional data have been investigated by Isakov,
Rundell, Prilepko, Orlovsky, and Vasin~\cite{Isakov2006, prilepko2000methods, Rundell1980}.
An inverse scattering problem related to source identification was studied by
Kedzierawski~\cite{Kedzierawski1993}.

Ngoma~\cite{ngoma2024well} established the well-posedness of the inverse problem
\eqref{inv1}--\eqref{integparab1}, together with higher regularity of the solution pair
$(u,h)$ in parabolic H\"older spaces, in the case where the zeroth-order coefficient
satisfies $c(t,x)=0$ in the operator $\mathscr{L}$ defined in~\eqref{operator1}, and using  \(f(t,x) = f(t)\) in~\eqref{inv1}  (or equivalently, using \(w(t,x)\equiv 1\) in the above discussion).  In addition,
numerical results were presented using a finite element discretization in space and a
backward Euler scheme in time. It was shown that the recovery of the source function
can be reduced to a Volterra integral equation of the first kind, which is ill-posed
\cite{DamirchiEtAl19}. A collocation method was then used to derive a system of linear
equations, and Tikhonov regularization combined with the Morozov discrepancy principle was
employed to obtain stable reconstructions.

Unlike most of the above results, which primarily address the Laplace operator either in one space
dimension or in higher dimensions with time-independent coefficients, the present work
considers a more general parabolic operator~\eqref{operator1} in dimension $d\ge 1$, with
coefficients depending on both space and time variables.

In this work, we extend the results of~\cite{ngoma2024well} by including a nonzero
lower-order coefficient $c(t,x)$ in the operator $\mathscr{L}$ defined in~\eqref{operator1} as well as the fixed \((t,x)\)-dependent factor \(w(t,x)\) in the source term. We
establish well-posedness and higher regularity of solutions in parabolic H\"older spaces.
Furthermore, we present two numerical examples using finite element discretizations in
space: one involving the recovery of a smooth source function and another
involving the recovery of a piecewise constant source with jump discontinuities.

Since the resulting linear system involves a well-conditioned operator, we also establish
theoretically and illustrate numerically the ineffectiveness of identity Tikhonov regularization
in this setting. On the other hand, for the same well-conditioned operator, we
show both theoretically and numerically that Tikhonov regularization combined with the
Generalized Singular Value Decomposition (GSVD) provides an effective denoising strategy.

Unlike in~\cite{ngoma2024well}, where reconstruction errors were measured using the root
mean square error, the present work evaluates all reconstruction errors using full
parabolic H\"older norms. This choice provides a significantly more rigorous assessment
of accuracy, as these norms incorporate not only pointwise errors in the reconstructed
functions, but also errors in their temporal and spatial derivatives together with
corresponding H\"older seminorms. As a result, the reported errors capture fine-scale
oscillations and derivative sensitivity induced by noise and discretization, offering a
conservative and robust measure of reconstruction quality.

The remainder of the paper is organized as follows.
Section~\ref{notation-assumptions} introduces the notation, assumptions,
parabolic H\"older spaces, and auxiliary results used throughout the
analysis. In Section~\ref{existence_uniqueness}, we establish the existence
and uniqueness of the inverse solution by reducing the problem to the
corresponding forward parabolic problem and deriving an explicit
reconstruction formula for the unknown source function. 
Section~\ref{continous_dependence} establishes continuous dependence of the
solution pair on the prescribed data, thereby completing the well-posedness
analysis in the relevant parabolic H\"older spaces.
In Section~\ref{higher_regularity}, we investigate higher-order regularity
and show how additional regularity of the coefficients, forcing term,
integral data, and auxiliary function $w$ is inherited by the reconstructed
solution pair.

Section~\ref{regularization} turns to the effect of regularization on the
discrete inverse problem. We first analyze why identity Tikhonov
regularization is ineffective when the resulting operator is perfectly
conditioned, and then show, using the generalized singular value
decomposition, how derivative-based regularization can nevertheless provide
effective denoising by selectively penalizing prescribed modes.
Finally, Section~\ref{numerical_results} develops the finite element and
time-discretization scheme and presents numerical experiments for both
smooth and piecewise constant source functions. The experiments illustrate
the theoretical findings and assess the reconstructions under increasing
levels of noise using full parabolic H\"older norms.

\section{Notation, Assumptions, and Preliminaries}\label{notation-assumptions}
\subsection{Notation}
\begin{definition}\label{def1}
Assume that $\Omega\subset\mathbb{R}^d$, $d\ge 1$, $\delta\in (0,1)$, and $k = 0, 1, 2,\cdots$. 
\begin{description}
\item[(i)] If $u:\Omega\to\mathbb{R}$ is bounded and continuous, we write $\mathlarger{|u|_{0;\;\Omega} = [u]_{0; \Omega} = \sup_{x\in \Omega}\abs{u(x)}}$
\item[(ii)] The $\delta^{th}$ H\"{o}lder seminorm of $u:\Omega\to\mathbb{R}$ is 
\[
\left[u\right]_{\delta;\;\Omega} = \mathop{\sup_{x\ne y}}_{x, y\in \Omega}\frac{|u(x) - u(y)|}{|x - y|^{\delta}},
\]
and the $\delta^{th}$ H\"{o}lder norm is $|u|_{\delta;\;\Omega} = |u|_{0;\;\Omega} + \left[u\right]_{\delta;\;\Omega}$.

The H\"{o}lder space $C^{\delta}(\Omega)$ consists of all functions $u\in C(\overline{\Omega})$ for which $|u|_{\delta;\;\Omega}<\infty$.
\item[(iii)] We set
\[
[u]_{k; \Omega} = \max_{|\beta| = k}|D^{\beta}u|_{0; \Omega},\quad |u|_{k;\Omega} = \sum_{j = 0}^k[u]_{j;\Omega},\quad [u]_{k+\delta; \Omega} = \max_{|\beta| = k}|D^{\beta}u|_{\delta; \Omega}.
\]
The H\"{o}lder space $C^{k+\delta}(\Omega)$ is the Banach space of all functions $u\in C^k(\Omega)$ for which the norm
\[
|u|_{k+\delta;\;\Omega} = |u|_{k;\Omega} + [u]_{k+\delta;\;\Omega}<\infty.
\]
\end{description}
\end{definition}

\begin{definition}\label{def2}
A point in $\mathbb{R}^{d+1}$ will be denoted by $z=(t,x)$. Define the parabolic distance between the points $z_1 = (t,x)$ and $z_2 = (s,y)$ as 
\[
\rho(z_1,z_2) = |x - y| + |t - s|^{1/2}.
\]
Fix $\delta\in (0,1)$. If $u\in Q_T\subset \mathbb{R}^{d+1}$, we define
\begin{align*}
|u|_{0;\;Q_T} & = \sup_{z\in Q_T}\abs{u(z)},\\
\left[u\right]_{\delta/2,\,\delta;\;Q_T} &= \mathop{\sup_{z_1\ne z_2}}_{z_i\in Q_T}\frac{|u(z_1) - u(z_2)|}{\rho^{\delta}(z_1,z_2)},\\
|u|_{\delta/2,\;\delta;\;Q_T} &= |u|_{0;\;Q_T} + \left[u\right]_{\delta/2,\,\delta;\;Q_T}.
\end{align*}
We denote by $C^{\delta/2,\,\delta}(Q_T)$ the H\"{o}lder space of all functions $u$ for which $|u|_{\delta/2,\;\delta;\;Q_T}<\infty$. For a real-valued function $u(z)$ defined in $Q_T$, let
\begin{align}
    \left[ u\right]_{1+\delta/2,\,2 + \delta;\,Q_T} &:=\left[\partial_tu\right]_{\delta/2,\,\delta,\,Q_T} + \sum_{i,j = 1}^d\left[ \partial_{x_ix_j}u\right]_{\delta/2,\,\delta;\,Q_T}, \notag\\
    |u|_{1+\delta/2,\,2 + \delta;\,Q_T} &:= |u|_{0;\,Q_T} + \sum_{i = 1}^d|\partial_{x_i}u|_{0;\,Q_T} + |\partial_tu|_{0;\,Q_T} + \sum_{i,j = 1}^d\left|\partial_{x_ix_j}u\right|_{0;\,Q_T} + \left[ u\right]_{1+\delta/2,\,2 + \delta;\,Q_T}.\label{full_norm}
\end{align}
The parabolic H\"{o}lder space $C^{1+\delta/2,\,2+\delta}(Q_T)$ is the set of real-valued functions $u(z)$ defined in $Q_T$ for which both $\left[ u\right]_{1+\delta/2,\,2 + \delta;\,Q_T}<\infty$ and $|u|_{1+\delta/2,\,2 + \delta;\,Q_T}<\infty$.

\end{definition}
For more on H\"{o}lder spaces, see, for example,~\cite{Friedman08, Gilbarg-Trudinger96, Krylov96, Ladyzhenskaia88}.

\subsection{Assumptions}\label{assumption}
We assume the following about the data
\begin{description}
\item [$(A1)$]  The coefficients $a^{ij}, b^i ,c\in C^{\delta/2,\delta}(\overline{Q}_T)$. Moreover, there are constants $\nu, \Lambda$ such that
\[
\nu |\xi|^2\le\sum_{i,j=1}^{d}a^{ij}(t,x)\xi_i\xi_j\le\Lambda |\xi|^2,\; \text{for all $\xi\in\mathbb{R}^d$, $(t,x)\in\overline{Q}_T$}
\]

\item [$(A2)$] The domain $\Omega$ is a bounded $C^{2+\delta}$ domain. The boundary $\partial\Omega$ of $\Omega$ is of class $C^{2 + \delta}$. The final time $T>0$ is finite.
\item [$(A3)$] The functions $\phi\in C^{\delta/2,\delta}(\overline{Q}_T)$,  $u_0\in C^{2 + \delta}(\overline{\Omega})$, $g \in C^{(1+\delta)/2, 1+\delta}(\overline{S}_T)$, and they satisfy the compatibility condition of order $\lfloor (1+\delta)/2\rfloor = 0$: $g(x,0) = \frac{\partial u_0(x)}{\partial \nu}, x\in\partial\Omega$, as stated in~\cite[Chapter 4]{Ladyzhenskaia88}, where $\lfloor\cdot\rfloor$ denotes the floor function.
\item [$(A4)$] The function $\mu \in C^{1+\delta/2}[0,T]$ and we assume the compatibility condition $\mu(0)=\int_\Omega u_0(x) dx$.
\end{description}

\subsection{Preliminaries}\label{prelim}
In this section we present several results that will be used in subsequent sections. The first result is a restatement of a more general theorem taken from~\cite[Theorem 5.3 pp. 320-321]{Ladyzhenskaia88} that deals with the well-posedness of the forward problem~\eqref{parab1}--\eqref{bcparab1} and that will be needed in the next section.
\begin{lemma}\label{keyTheorem}
Suppose $\delta\in (0,1)$ and Assumptions $(A1)$-$(A2)$ hold. Then for any $\phi, u_0, g$ satisfying Assumption $(A3)$, 
the forward problem~\eqref{parab1}--\eqref{bcparab1} has a unique solution $u\in C^{1 + \delta/2, 2+\delta}(\overline{Q}_T)$, with
\begin{equation*}
    \left|u\right|_{1+\delta/2,2+\delta;Q_T}\le C\left(\left|\phi\right|_{\delta/2, \delta;Q_T} + \left|u_0\right|_{2+\delta; \Omega} + \left|g\right|_{(1+\delta)/2, 1+\delta; S_T}\right),
\end{equation*}
where $C>0$ is a constant not depending on $\phi, u_0$ and $g$.
\end{lemma}

The second result is the following lemma that will be useful in establishing estimates leading to the proof of the mathematical continuous dependence of the solutions on the data.
\begin{lemma}\label{norm_equiv}
Let $u$ be a function in a convex or cylindrical domain $Q_T = \Omega\times (0,T)\subset\mathbb{R}^{d+1}$, and let $\delta\in (0,1)$.
Define
\[
\left[u \right]'_{\delta/2, \delta;Q_T} = \mathop{\sup_{t\ne s}}_{(t,x),(s,x)\in Q_T}\frac{\left|u(t,x) - u(s,x) \right|}{|t-s|^{\delta/2}} + \mathop{\sup_{x\ne y}}_{(t,x),(t,y)\in Q_T}\frac{\left|u(t,x) - u(t,y) \right|}{|x-y|^{\delta}}.
\]
Then $\left[u \right]'_{\delta/2, \delta;Q_T}\le 2\left[u \right]_{\delta/2, \delta;Q_T}$, where $\left[u \right]_{\delta/2, \delta;Q_T}$ is defined in Definition~\ref{def2}.
\end{lemma}

\begin{proof}
The proof is a simple consequence of the fact that
\begin{align*}
\frac{\left|u(t,x) - u(s,x) \right|}{|t - s|^{\delta/2}} + \frac{\left|u(t,x) - u(t,y) \right|}{|x - y|^{\delta}}
\le 2\mathop{\sup_{(t,x)\ne (s,y)}}_{(t,x),(s,y)\in Q_T}\frac{\left|u(t,x) - u(s,y) \right|}{\rho^{\delta}\left((t,x),(s,y)\right)} 
\end{align*}
\end{proof}


The next lemma provides a compact standard quantitative H\"{o}lder-operational-bounds needed in the stability theorem. These are facts from H\"{o}lder continuous spaces, so we omit the proof.
\begin{lemma}\label{smooth-Holder-operational-closure} 
Let $\delta\in(0,1)$ and set
\[
X:=C^{1+\delta/2}[0,T],
\qquad
Y:=C^{1+\delta/2,2+\delta}(\overline{Q}_T).
\]
Then the following statements hold.

\begin{enumerate}
    \item If $r\in Y$ and
    \[
    R(t):=\int_\Omega r(t,x)\,dx,
    \]
    then $R\in X$ and
    \[
    |R|_{1+\delta/2;\,(0,T)}
    \le C_\Omega
    |r|_{1+\delta/2,\,2+\delta;\,Q_T}.
    \]

    \item If $p,q\in X$, then $pq\in X$ and
    \[
    |pq|_{1+\delta/2;\,(0,T)}
    \le C
    |p|_{1+\delta/2;\,(0,T)}
    |q|_{1+\delta/2;\,(0,T)}.
    \]

    \item If $p\in X$ and $r\in Y$, then the function $Q(t,x):=p(t)r(t,x)$
    belongs to $Y$, and
    \[
    |Q|_{1+\delta/2,\,2+\delta;\,Q_T}
    \le C
    |p|_{1+\delta/2;\,(0,T)}
    |r|_{1+\delta/2,\,2+\delta;\,Q_T}.
    \]

    \item Let $p\in X$ and assume that
    \[
    \inf_{t\in[0,T]}|p(t)|\ge \eta>0.
    \]
    Then $p^{-1}\in X$, and
    \[
    |p^{-1}|_{1+\delta/2;\,(0,T)}
    \le C\left(\eta,|p|_{1+\delta/2;\,(0,T)}\right).
    \]
\end{enumerate}
Here the constants depend only on the indicated quantities, $\Omega$,
$T$, and $\delta$.
\end{lemma}

\section{Existence and uniqueness}\label{existence_uniqueness}
 In this section we present the existence and uniqueness of solutions to suitable modifications of the inverse problem~\eqref{inv1}--\eqref{integparab1}.  We will first prove the existence and uniqueness of solutions to the following inverse problem: 

Find $(u,h)$ such that
\begin{alignat}{2}
\partial_{t}{u}(t,x) + \mathscr{L}u(t,x) &= h'(t)w(t,x)+h(t)\bigl(\partial_t w(t,x)+\mathscr{L}w(t,x)\bigr) + \phi(t,x), &&\quad (t,x) \in Q_T \label{inv2} \\
u(0,x) &= u_{0}(x), &&\quad x \in\Omega \label{icinv2} \\
\frac{\partial u(t,x)}{\partial \nu} &= h(t) \frac{\partial w(t,x)}{\partial \nu}+g(t,x), &&\quad (t,x) \in S_T \label{bcinv2}\\
\int_{\Omega}u(t,x)\,dx &= \mu(t). &&\quad t \in (0,T] \label{integparab2}
\end{alignat}
After the proof, in Remark~\ref{simpler-forms}, we provide some special cases that lead to simpler forms than the general form presented above.

\begin{theorem}\label{existence-uniqueness}
  Let $\delta\in (0,1)$ and suppose that assumptions $(A1)$-$(A4)$ hold.  Let $v=v(t,x)$ be the unique solution of the forward problem~\eqref{parab1}--\eqref{bcparab1}.  Let $w\in C^{1+\delta/2,2+\delta}\left(\overline{Q}_T\right)$ be a given function such that $\int_\Omega w(t,x)\,dx\neq 0$ for all $0\le t\le T$. 
  Then the inverse problem~\eqref{inv2}--\eqref{integparab2} has the following solution $(u,h)\in C^{1+\delta/2, 2+\delta}(\overline{Q}_T)\times C^{1+\delta/2}[0,T]$:  
 \begin{equation}
      u(t,x)=h(t)w(t,x)+v(t,x),\label{sol_u}
  \end{equation}
  where
 \begin{equation}
     h(t)=\frac{\mu(t)-\int_\Omega v(t,x)\,dx}{\int_\Omega w(t,x)\,dx}\label{sol_h}.\\
 \end{equation}
  
Under the additional condition that $h(0)=0$, this solution is unique.    
 \end{theorem}

\begin{proof}  Let \(u\) and \(h\) be defined as in equations~\eqref{sol_u} and~\eqref{sol_h}. First, \(\mu(0)-\int_\Omega v(0,x)\ dx=0\) because \(v\) satisfies~\eqref{icparab1} and \(\mu\) satisfies the compatibility condition in assumption $(A4)$. Therefore, equation~\eqref{sol_h} implies \(h(0)=0\) and equation~\eqref{sol_u} implies \(u(0,x)=v(0,x)=u_0(x)\).

Now we compute from~\eqref{sol_u} the relevant derivatives in terms of $h$ and $w$: 
\begin{align*}
\partial_t u(t,x) &=  h^\prime(t)w(t,x) + h(t)\partial_t w(t,x)+\partial_t v(t,x),\\
\mathscr{L} u(t,x) &= h(t) \mathscr{L}w(t,x) + \mathscr{L} v(t,x),\\
\frac{\partial u(t,x)}{\partial\nu}&=h(t)\frac{\partial w(t,x)}{\partial\nu}+\frac{\partial v(t,x)}{\partial\nu}.
\end{align*}
  Putting all of this together, \(u\) and \(h\) satisfy the equations
\begin{alignat}{2}
\partial_{t}{u}(t,x) + \mathscr{L}u(t,x) &=  h'(t)w(t,x)+h(t)\bigl(\partial_t w(t,x)+\mathscr{L}w(t,x)\bigr) + \phi(t,x), &&\quad (t,x) \in Q_T \label{fwd1} \\
u(0,x) &= u_{0}(x), &&\quad x \in\Omega \label{icfwd1} \\
\frac{\partial u(t,x)}{\partial \nu} &= h(t)\frac{\partial w(t,x)}{\partial\nu}+g(t,x), &&\quad (t,x) \in S_T. \label{bcfwd1}
\end{alignat}

 Now we take any $t\in (0,\,T]$ and compute the integral of $u(t,x)$ over $\Omega$:

\[
    \int_\Omega u(t,x)\, dx = \left[ \frac{\mu(t)-\int_\Omega v(t,x)\,dx}{\int_\Omega w(t,x)\,dx}\right]\int_\Omega w(t,x)\,dx+\int_\Omega v(t,x)\, dx = \mu(t),
\]
so that $u$ satisfies integral condition~\eqref{integparab2}.  Therefore, we have verified that the choice $(u,h)$ specified in equations~\eqref{sol_u}-\eqref{sol_h} solves the inverse problem~\eqref{inv2}--{}\eqref{integparab2}.

Next, we show that \(h\in C^{1+\delta/2}[0,T]\) and \(u\in C^{1+\delta/2,2+\delta}(\overline{Q}_T)\). To this end, recall that \(v\in C^{1+\delta/2,2+\delta}(\overline{Q}_T)\) according to Lemma~\ref{keyTheorem}.  Both \(t\mapsto \int_\Omega w(t,x)dx \) and \(t\mapsto \int_\Omega v(t,x)dx \) are in \(C^{1+\delta/2}[0,T]\) by Lemma~\ref{smooth-Holder-operational-closure} (Statement 1).  Since we assumed \(\int_\Omega w(t,x)dx\) is nonzero, the reciprocal \(t\mapsto \left(\int_\Omega w(t,x)dx\right)^{-1}\) is in \(C^{1+\delta/2}[0,T]\) by Lemma~\ref{smooth-Holder-operational-closure} (Statement 4).  Since we assumed \(\mu\in C^{1+\delta/2}[0,T]\), it follows from linear closure and Lemma~\ref{smooth-Holder-operational-closure} (Statement 2) that \(h\in C^{1+\delta/2}[0,T]\).  

In addition, to verify that \(h\) is given by~\eqref{sol_h}, we compute
\begin{align*}
    \mu(t)&=\int_\Omega u(t,x)\, dx
    = h(t)\int_\Omega w(t,x)\,dx+\int_\Omega v(t,x)\,dx.
\end{align*}
Solving this expression for \(h(t)\) verifies~\eqref{sol_h}.
Finally, \(u=hw+v\in C^{1+\delta/2,2+\delta}(\overline{Q}_T)\) by Lemma~\ref{smooth-Holder-operational-closure} (Statement 3) and linear closure.

For the stated uniqueness property, let \((u,h)\) be any solution to the inverse problem~\eqref{inv2}--{}\eqref{integparab2} such that $h(0)=0$.  First, notice that \(u\) solves the forward problem represented by~\eqref{fwd1}--{}\eqref{bcfwd1}.  The same may be said for the function defined by the expression  \(h(t)w(t,x)+v(t,x)\), with the help of the assumption \(h(0)=0\) for the initial condition.  Therefore, the difference \(u(t,x)-h(t)w(t,x)-v(t,x)\) satisfies the homogeneous problem associated with ~\eqref{fwd1}--{}\eqref{bcfwd1}.  By uniqueness of solutions to the forward problem, \(u(t,x)-h(t)w(t,x) -v(t,x)\equiv 0\), and $u$ satisfies~\eqref{sol_u}.
\end{proof}

\begin{remark}\label{solution-set}
  In the situation of Theorem~\ref{existence-uniqueness}, uniqueness depends on the restriction \(h(0)=0\).  If we remove this assumption and assume that \((\tilde{u},\tilde{h})\) is a solution, then take \(\tilde{v}(t,x)=\tilde{u}(t,x)-\tilde{h}(t)w(t,x)\).  Then \(\tilde{v}\) solves the forward problem~\eqref{parab1}--\eqref{bcparab1}, but with initial condition \(\tilde{v}(0,x)=u_0(x)-\tilde{h}(0)w(0,x)\).  Using the integral condition~\eqref{integparab2}, we must have
  \[
  \tilde{h}(t)=\frac{\mu(t)-\int_\Omega \tilde{v}(t,x)\,dx}{\int_\Omega w(t,x)\,dx}
  \]
  and \(\tilde{u}(t,x)=\tilde{h}(t)w(t,x)+\tilde{v}(t,x)\).  It is easy to check that conversely, these choices for \(\tilde{v}, \tilde{u},\,\text{and } \tilde{h}\) will yield a solution to the inverse problem~\eqref{inv2}--\eqref{integparab2}, for any given value of \(\tilde{h}(0)\), and also that \(\tilde{h}\in C^{1+\delta/2}[0,T]\) and \(\tilde{u}\in C^{1+\delta/2,2+\delta}(\overline{Q}_T)\).  We emphasize that the solution form \(u(t,x)=h(t)w(t,x)+v(t,x)\) already implies \(h(0)=0\), and that other values of \(h(0)\) result in \(u(t,x)={h}(t)w(t,x)+\tilde{v}(t,x)\).
\end{remark}

\begin{remark}\label{simpler-forms}  By direct computation as in the proof of Theorem~\ref{existence-uniqueness}, the form \(u(t,x)=h(t)w(t,x)+v(t,x)\) alone implies that the inverse problem must have the terms of the form presented in~\eqref{inv2}--{}\eqref{integparab2}.  Therefore, to arrive at a simpler formulation one could impose additional conditions on \(w(t,x)\) to simplify the form of the expressions in~\eqref{inv2}--{}\eqref{integparab2}.  Here are a few practical examples of such conditions, though of course others are possible.  
\begin{itemize}
    \item If $w\equiv 1$ and if $c\equiv 0$ in $\mathscr{L}$ (see~\eqref{operator1}), we recover exactly the result in~\cite{ngoma2024well}. In the numerical solutions in Section~\ref{piece_wise_approximation}, we use $w\equiv 1$, $c\equiv 1$, and reconstruct a piecewise constant source $h(t)$ with jump discontinuities in a PDE with forcing term $h(t)+h'(t)$.
\item If $\partial_t w+\mathscr{L}w=0$ with homogeneous Neumann boundary conditions on $w$, then the inverse problem is
\begin{alignat*}{2}
\partial_{t}{u}(t,x) + \mathscr{L}u(t,x) &= h'(t)w(t,x) + \phi(t,x), &&\quad (t,x) \in Q_T  \\
u(0,x) &= u_{0}(x), &&\quad x \in\Omega \\
\frac{\partial u(t,x)}{\partial \nu} &= g(t,x), &&\quad (t,x) \in S_T \\
\int_{\Omega}u(t,x)\,dx &= \mu(t). &&\quad t \in (0,T]
\end{alignat*}

\item If $w=w(x)$ with homogeneous Neumann boundary conditions on $w$, then we have the inverse problem
\begin{alignat*}{2}
\partial_{t}{u}(t,x) + \mathscr{L}u(t,x) &= h'(t)w(x) + h(t)\mathscr{L}w(x)+ \phi(t,x), &&\quad (t,x) \in Q_T  \\
u(0,x) &= u_{0}(x), &&\quad x \in\Omega \\
\frac{\partial u(t,x)}{\partial \nu} &= g(t,x), &&\quad (t,x) \in S_T \\
\int_{\Omega}u(t,x)\,dx &= \mu(t). &&\quad t \in (0,T]
\end{alignat*}

\item In case $w$ is chosen to have support in a subset $\Omega_0\subset\subset\Omega$, note that in $\Omega\setminus \Omega_0$, the solution $u$ and the data of the inverse problem agree with $v$ and the data of forward problem, respectively.  This shows that it is possible to satisfy the integral condition by acting only in a desired subset of $\Omega$. 
\end{itemize}
\end{remark}

\section{Continuous Dependence on the Data}\label{continous_dependence}
In this section, we show that the solution of the inverse problem depends continuously on the data.

\begin{theorem}\label{stability}
Let $\delta\in (0,1)$. Suppose Assumptions $(A1)$-$(A4)$ hold. According to Theorem~\ref{existence-uniqueness}, let $(u,h)$ and $(\tilde{u},\tilde{h})$ be the unique solutions of the inverse problem~\eqref{inv2}--\eqref{integparab2} corresponding to the sets of data $\{\phi, u_0, g, \mu\}$ and $\{\tilde{\phi}, \tilde{u}_0, \tilde{g}, \tilde{\mu}\}$, respectively, and satisfying $h(0)=\widetilde h(0)=0$. Then there exist constants $C_1,C_2>0$, depending only on the fixed
problem parameters, including $\Omega$, $T$, $\delta$, the coefficients
of $\mathscr{L}$, the function $w$, and the positive lower bound
\[
\eta:=\min_{t\in[0,T]}
\left|\int_\Omega w(t,x)\,dx\right|,
\]
but independent of the two sets of data
$\{\phi,u_0,g,\mu\}$ and
$\{\widetilde\phi,\widetilde u_0,\widetilde g,\widetilde\mu\}$,
such that
\begin{align*}
    \left|h-\tilde{h}\right|_{1+\delta/2;(0,T)}
    &\le C_1\left(\left|\mu - \tilde{\mu}\right|_{1+\delta/2;(0,T)} +\left|\phi-\tilde{\phi}\right|_{\delta/2, \delta;Q_T} + \left|u_0 - \tilde{u}_0\right|_{2+\delta; \Omega} + \left|g - \tilde{g}\right|_{(1+\delta)/2, 1+\delta; S_T}\right),
\end{align*}
and
\begin{align*}
    \left|u - \tilde{u}\right|_{1+\delta/2, 2 + \delta; Q_T} &\le
    C_2\left(\left|\mu - \tilde{\mu}\right|_{1+\delta/2; (0,T)} + \left|\phi - \tilde{\phi}\right|_{\delta/2, \delta;Q_T} + \left|u_0 - \tilde{u}_0\right|_{2+\delta; \Omega} + \left|g - \tilde{g}\right|_{(1+\delta)/2, 1+\delta; S_T}\right).
\end{align*}
\end{theorem}
\begin{proof}
Let $v$ and $\tilde{v}$ be the forward solutions corresponding to the data $(\phi,u_0,g)$ and $(\tilde\phi,\tilde{u}_0,\tilde{g})$, respectively.  Now we use~\eqref{sol_h} to obtain
\[
h(t)=\frac{\mu(t)-\int_\Omega v(t,x)\,dx}{q(t)},
\qquad
\tilde{h}(t)=\frac{\tilde\mu(t)-\int_\Omega \tilde{v}(t,x)\,dx}{q(t)},
\]
where
\(
q(t):=\int_\Omega w(t,x)\,dx
\).
Hence
\begin{equation}\label{h_difference}
h(t)-\tilde{h}(t)
=
q(t)^{-1}
\left[
\mu(t)-\tilde\mu(t)
-
\int_\Omega (v(t,x)-\tilde{v}(t,x))\,dx
\right].
\end{equation}

Since $w\in C^{1+\delta/2,2+\delta}(\overline{Q}_T)$, Lemma~\ref{smooth-Holder-operational-closure} (Statement 1) implies that
\(q\in C^{1+\delta/2}[0,T]\). In addition, since $q(t)\neq 0$ for every $t\in[0,T]$, there exists $\eta>0$ such that
\[
\inf_{t\in[0,T]}|q(t)|\ge \eta.
\]
Hence, Statement~4 of Lemma~\ref{smooth-Holder-operational-closure}
implies that $q^{-1}\in C^{1+\delta/2}[0,T]$.

Moreover, Lemma~\ref{smooth-Holder-operational-closure} (Statement 1) shows that
\(
t\mapsto \int_\Omega (v(t,x)-\tilde{v}(t,x))\, dx
\) belongs to $C^{1+\delta/2}[0,T]$, and its Hölder norm is controlled by
\(|v-\tilde{v}|_{1+\delta/2,2+\delta;Q_T}
\). 

Applying Lemma~\ref{smooth-Holder-operational-closure} to~\eqref{h_difference} therefore gives
\begin{align}
|h-\tilde{h}|_{1+\delta/2;(0,T)}
&\le
C_1\Big(
|\mu-\tilde\mu|_{1+\delta/2;(0,T)}
+
|v-\tilde{v}|_{1+\delta/2,2+\delta;Q_T}
\Big),
\label{h_estimate}
\end{align}
where $C_1$ may depend on $|q^{-1}|_{1+\delta/2;(0,T)}$, and hence on the fixed function $w$ and the lower bound $\eta$, but not on either set of perturbed data.

Next, $v$ and $\tilde{v}$ satisfy the forward problems
\begin{align*}
\partial_t v+\mathscr L v &= \phi, &
v(0,x)&=u_0, &
\frac{\partial v}{\partial\nu}&=g,\\
\partial_t \tilde{v}+\mathscr L \tilde{v} &= \tilde\phi, &
\tilde{v}(0,x)&=\tilde{u}_0, &
\frac{\partial \tilde{v}}{\partial\nu}&=\tilde{g}.
\end{align*}
Hence the difference $v-\tilde{v}$ satisfies
\begin{align*}
\partial_t (v-\tilde{v})+\mathscr L (v-\tilde{v})
&=
\phi-\tilde\phi
\quad \text{in } Q_T,\\
(v-\tilde{v})(0,x)
&=
u_0-\tilde{u}_0
\quad \text{in } \Omega,\\
\frac{\partial (v-\tilde{v})}{\partial\nu}
&=
g-\tilde{g}
\quad \text{on } S_T.
\end{align*}

By Lemma~\ref{keyTheorem}, it follows that
\begin{align}
|v-\tilde{v}|_{1+\delta/2,2+\delta;Q_T}
\le
C_2\Big(
|\phi-\tilde\phi|_{\delta/2,\delta;Q_T}
+
|u_0-\tilde{u}_0|_{2+\delta;\Omega}
+
|g-\tilde{g}|_{(1+\delta)/2,1+\delta;S_T}
\Big).
\label{v_estimate}
\end{align}

Substituting~\eqref{v_estimate} into~\eqref{h_estimate} yields
\begin{align}
|h-\tilde{h}|_{1+\delta/2;(0,T)}
\le
C_3\Big(
|\mu-\tilde\mu|_{1+\delta/2;(0,T)}
+
|\phi-\tilde\phi|_{\delta/2,\delta;Q_T}
+
|u_0-\tilde{u}_0|_{2+\delta;\Omega}
+
|g-\tilde{g}|_{(1+\delta)/2,1+\delta;S_T}
\Big).
\label{h_final}
\end{align}

Finally, from equation~\eqref{sol_u},
\[
u(t,x)-\tilde{u}(t,x)
=
v(t,x)-\tilde{v}(t,x)
+
\left(h(t)-\tilde{h}(t)\right)w(t,x).
\]
Since $w\in C^{1+\delta/2,2+\delta}(\overline{Q}_T)$, Lemma~\ref{smooth-Holder-operational-closure} (statement 3) implies that $(h-\tilde{h})w\in C^{1+\delta/2,2+\delta}(\overline{Q}_T),$
and
\[
|(h-\tilde{h})w|_{1+\delta/2,2+\delta;Q_T}
\le
C_4\,|h-\tilde{h}|_{1+\delta/2;(0,T)},
\]
where $C_4$ may depend on $|q^{-1}|_{1+\delta/2;(0,T)}$, and hence on the fixed function $w$ and the lower bound $\eta$, but not on either set of perturbed data.
Therefore
\[
|u-\tilde{u}|_{1+\delta/2,2+\delta;Q_T}
\le
|v-\tilde{v}|_{1+\delta/2,2+\delta;Q_T}
+
C_4\,|h-\tilde{h}|_{1+\delta/2;(0,T)}.
\]
Combining this with~\eqref{v_estimate} and~\eqref{h_final} yields the desired estimate for $u$.
\end{proof}
\section{Higher Regularity of the Solutions} \;\label{higher_regularity}
In this section, we show that if $\mu$ and \(w\) possess some higher regularity, then the solution pair $(u,h)$ of the inverse problem will possess higher regularity as well.  We start with the following result~\cite[Theorem 11, Chapter 3, Section 5]{Friedman08}, which ensures that as long as the source function \(\phi\) and the coefficients \(a^{ij}\), \(b^{i}\), and \(c\) possess some higher regularity, then the forward solution \(v\) will also have this regularity, along with all the derivatives \(\partial_t v\), \(\partial_{x_i} v\) and \(\partial_{x_ix_j}v\).  
\begin{lemma}\label{friedman}
    Suppose \(\delta\in (0,1)\) and Assumptions (A1)-(A2) hold.  Let \(M\) and \(N\) be nonnegative integers.  Assume that
    \[
D^\beta_x\partial_t^k a^{ij},\quad D^\beta_x\partial_t^k b^{i}, \quad D^\beta_x\partial_t^k c,\quad D^\beta_x\partial_t^k \phi,\quad (0\le \abs{\beta}+2k\le M,\; k\le N)
    \]
exist and are in \(C^{\delta/2,\delta}({Q}_T)\).  If \(v\) is a solution to the forward problem~\eqref{parab1} in $Q_T$, then
\[
D^\beta_x\partial_t^k v\in C^{\delta/2,\delta}({Q}_T),\quad (0\le \abs{\beta}+2k\le M+2,\; k\le N+1).
\]
\qed
\end{lemma}

Since the regularity of \(v\) is therefore guaranteed, our task amounts to verifying that a smoother version of Lemma~\ref{smooth-Holder-operational-closure} holds. We omit the proof as these results are standard.
\begin{lemma}~\label{smoother-Holder-operational-closure}
Let \(\delta\in(0,1)\) and that Assumption (A2) holds. Let \(p\), \(q\in C^{\delta/2}[0,T]\), and let \(r\in C^{\delta/2,\delta}(\overline{Q}_T)\).  Let $N$ be a given positive integer and assume that the derivatives \(p^{(k)},\, q^{(k)} \in C^{\delta/2}[0,T]\), and  \(\partial_t^k r\in C^{\delta/2,\delta}(\overline{Q}_T)\) for all \(0\le k\le N\).
\begin{enumerate}
    \item If \(R(t) := \int_\Omega r(t,x)\,dx\), then \( R^{(k)}\in C^{\delta/2}[0,T]\) for all \(0\le k\le N\).
    \item If \(p(t)\neq 0\) for all \(t\in [0,T]\) and  \(P(t) := p(t)^{-1}\), then  \( P^{(k)}\in C^{\delta/2}[0,T]\) for all \(0\le k\le N\). 
    \item If \( S(t):= p(t)q(t)\), then \( S^{(k)}\in C^{\delta/2}[0,T]\) for all \(0\le k\le N\).
\end{enumerate}
\end{lemma}

Now we prove the higher regularity of the solutions.
\begin{theorem}\label{regularity}
Suppose Assumptions (A1)-(A2) hold.  Assume that \(a^{ij},b^i, c, \phi\) satisfy the hypotheses of Lemma~\ref{friedman}, that \(v\) solves the forward problem~\eqref{parab1}, and that
\[
\partial_t^k v\in C^{\delta/2,\delta}(\overline{Q}_T),
\qquad
0\le k\le N+1.
\]  
In addition, assume that $w$ and $\mu$ satisfy
\[
D^\beta_x\partial_t^k w\in C^{\delta/2,\delta}(\overline{Q}_T),\quad \mu^{(k)}\in C^{\delta/2}[0,T] \quad (0\le \abs{\beta}+2k\le M+2, \; 0\le k\le N+1),
\]
and that \(\int_\Omega w(t,x)dx\neq 0\) for all \(t\in[0,T]\).  Define \(h\) and \(u\) by the equations~\eqref{sol_u} and~\eqref{sol_h} used in Theorem~\ref{existence-uniqueness} to solve the inverse problem~\eqref{inv2}--~\eqref{integparab2}.  Then
\[
D^\beta_x\partial_t^k u\in C^{\delta/2,\delta}({Q}_T), \quad (0\le \abs{\beta}+2k\le M+2,\; k\le N+1),
\]
and
\[
h^{(k)}\in C^{\delta/2}[0,T], \quad 0\le k\le N+1.
\]
\end{theorem}

\begin{proof}
By Lemma~\ref{friedman}, \(D^\beta_x\partial_t^k v\in C^{\delta/2,\delta}(\overline{Q}_T)\) for all \( (0\le \abs{\beta}+2k\le M+2, k\le N+1)\).  For \(h\) we take equation~\eqref{sol_h} one operation at a time.  We are checking that all derivatives up to order \(N+1\) are in  \(C^{\delta/2}[0,T]\).  This smoothness holds for the functions \(t\mapsto\int_\Omega v(t,x) dx\) and \(t\mapsto\int_\Omega w(t,x) dx\) by Lemma~\ref{smoother-Holder-operational-closure} (Statement 1), and then for \(t\mapsto\left(\int_\Omega w(t,x) dx\right)^{-1}\) by Statement 2 of the same Lemma.  The smoothness of \(h\) then follows from linearity and Statement 3  of Lemma~\ref{smoother-Holder-operational-closure}.

For \(u(t,x)=h(t)w(t,x)+v(t,x)\), we can write
\[
D^\beta_x\partial_t^k(hw) = \sum_{j=0}^k\binom{k}{j}h^{(j)}(t)D^\beta_x\partial_t^{k-j}w(t,x).
\]
Since every term to the right-hand side has the desired H\"{o}lder regularity under our assumptions, we are done.
\end{proof}

\begin{remark}
Friedman's Theorem~11 provides sufficient conditions for the corresponding
interior regularity of the forward solution $v$ in $Q_T$. Higher regularity
up to the parabolic boundary, however, generally requires corresponding
higher regularity of the initial and boundary data together with the
appropriate higher-order compatibility conditions.

In the present inverse problem, additional regularity of $v$ up to
$\overline{Q}_T$ is required in the time variable because the reconstruction
formula for $h$ contains the spatial integral
\[
t\longmapsto \int_\Omega v(t,x)\,dx.
\]
Although Friedman's theorem guarantees the existence and H\"older continuity
of the higher derivatives of $v$ at interior points of $Q_T$, interior
regularity alone does not provide sufficient control of these derivatives as
$x\to\partial\Omega$. We therefore assume, in addition, that
\[
\partial_t^k v\in C^{\delta/2,\delta}(\overline{Q}_T),
\qquad 0\le k\le N+1.
\]
This assumption permits differentiation under the spatial integral and yields
\[
\frac{d^k}{dt^k}\int_\Omega v(t,x)\,dx
=
\int_\Omega \partial_t^k v(t,x)\,dx,
\qquad 0\le k\le N+1,
\]
with the resulting functions possessing the required temporal H\"older
regularity.
\end{remark}

\section{Regularization and Denoising for the Discrete Inverse Problem}
\label{regularization}
In this section, we illustrate why identity Tikhonov regularization is ineffective in the case of a well-conditioned operator. Then we show that combining Tikhonov regularization with the GSVD provides a denoising strategy.

To this end, notice that~\eqref{sol_h} can be written as
\[
h(t)\int_\Omega w(t,x)\,dx = \mu(t) - \int_\Omega v(t,x)\,dx.
\]
This in turn can be written in operator form as
\begin{equation}\label{matrixeq}
(Ah)(t) = H(t), \quad 0\le t<T,
\end{equation}
where
\[
(Ah)(t):= m(t)h(t),\quad m(t):= \int_\Omega w(t,x)\,dx, \quad H(t) = \mu(t) - \int_{\Omega}v(x,t)\,dx,
\]
and $A: = m(t)I$, where $I$ is the identity matrix. If $m(t) = m = \mathrm{Cte}$, then this shows that system~\eqref{matrixeq} is exact and well conditioned. The matrix $A$ is the identity matrix up to the scaling factor $m$. 

In practice, the data is almost always noisy. So a random noise will be added to the right-hand side of~\eqref{matrixeq} for simulation. If $A$ is ill-conditioned or singular, or the system is underdetermined, small noise can blow up in the naïve least-squares solution.

Tikhonov regularization addresses this by solving the optimization problem
\begin{equation}\label{TikhonovReg}
h_\alpha = \argmin_{h^*}\{\Vert Ah^*-H^{\epsilon}\Vert_2^2 + \alpha \Vert Lh^*\Vert_2^2\}
\end{equation}
where $H^{\epsilon} = H + \epsilon$, and $\epsilon$ is a random noise with normal distribution with mean $E[\epsilon]$ zero and standard deviation $\sigma_\epsilon$
\begin{equation}\label{std}
E[\epsilon] = 0,\quad \sigma_\epsilon = p\times\max_{t\in [0,T]}|H(t)|,
\end{equation}
and $p$ is the percentage of noise. 

The regularizer $L$ is either the identity matrix $I$ or a differential regularization matrix obtained by discretizing a differential operator. The residual/misfit term $\|Ah^*-H^\epsilon\|$ enforces consistency with the data while the penalty/prior term $\|Lh^*\|$ encodes what one prefers among all solutions consistent with the data: small norm ($L=I$), smooth (if $L$ approximates a derivative). The regularization parameter $\alpha$ balances data fit versus prior.

The Tikhonov regularized solution to the minimization problem~\eqref{TikhonovReg} is given by
\begin{equation}\label{TikhonovSol}
h_{\alpha} = \left(A^TA + \alpha(L^TL)\right)^{-1}A^TH^{\epsilon},
\end{equation}
where the superscript $T$ denotes the transpose operator. The estimation of the regularization parameter $\alpha$ is not a simple problem. Our simulations use the discrepancy principle. Numerically, it consists on finding $\alpha$ such that the norm of the residual $\Vert Ah_{\alpha}-H^{\epsilon}\Vert_2$ and the noise level $\Vert \epsilon\Vert_2$ are equal~\cite{morozov1966solution, Vogel02}. Graphically, $\alpha$ is the intersection of the graphs of $\Vert Ah_{\alpha}-H^{\epsilon}\Vert_2$ and $\Vert \epsilon\Vert_2$ plotted against values of $\alpha$ in a log-log scale. 

\subsection{Limitations of Tikhonov regularization with $L=I$ for well-conditioned operators}
We show why Tikhonov regularization with $L=I$ does not help when the condition number of the matrix $A$ is $\kappa(A) = 1$. The result is as follows.

\begin{proposition}\label{prop:Tikhonov_with_L=I}
Let $A\in\mathbb{R}^{m\times n}$ satisfy $\kappa(A)=1$. 
Then Tikhonov regularization with $L=I$ applies the same shrinkage
factor to all singular modes of $A$. Consequently, it cannot selectively
suppress rough or noise-dominated components while preserving other
components of the solution.
\end{proposition}

\begin{proof}
Since $A\in\mathbb{R}^{m\times n}$, it can be expressed in terms of its Singular Value Decomposition (SVD) as $A = U\Sigma V^T$, where $U\in\mathbb{R}^{m\times m}$ and $V\in\mathbb{R}^{n\times n}$ are orthogonal matrices of left and right singular vectors $u_i$ and $v_i$, respectively, and $\Sigma\in\mathbb{R}^{m\times n}$ is a diagonal matrix with nonnegative entries, $\Sigma = \diag(\sigma_1,\cdots,\sigma_r,0,\cdots)$, where $\sigma_1\ge\cdots\ge\sigma_r>0$ are singular values of $A$ and $r = \rank(A)$. Each triple $(\sigma_i, u_i, v_i)$ describes one ``mode" (or channel) of action $Av_i = \sigma_iu_i$. So one can think of the $i$-th mode as the ``input direction $v_i$'' mapping to ``output direction $u_i$'' scaled by $\sigma_i$.

Using the SVD formula, the fact that the right singular vectors $\{v_j\}$ form an orthonormal basis for $\mathbb{R}^n$, $H = Ah^*$, with $h^*$ the true solution, and $U$ is a matrix with orthonormal columns, we can write 
\[
h^* = \sum_{i=1}^n(v^T_ih^*)v_i,\quad\text{and}\quad u^T_iH = \sigma_i(v^T_ih^*).
\]
In this case, the least-squares solution (LS) to~\eqref{TikhonovSol} (with $\alpha = 0$) is
\begin{equation}\label{LS}
h_{LS} = \sum_{i = 1}^r\frac{1}{\sigma_i}\left[(u^T_i H) + (u^T_i\epsilon) \right]v_i = \sum_{i = 1}^r\left[\left(v^T_i h^*\right) + \frac{1}{\sigma_i}(u^T_i\epsilon) \right]v_i.
\end{equation}
Observe that if $\sigma_i$ is small, the inversion formula~\eqref{LS} multiplies random noise in those directions by a huge factor $1/\sigma_i$. The noisy solution fluctuate wildly from sample to sample, giving rise to variance, which comes from the propagation of random noise $\epsilon$ through the inverse operator (noise amplification). Tikhonov regularization~\eqref{TikhonovSol} replaces the unbounded filter factor $1/\sigma_i$ by the bounded filter factor $\theta_i(\alpha): = \sigma_i/(\sigma_i^2+\alpha)$:
\begin{equation}\label{Tik}
    h_\alpha = \sum_{i = 1}^r\frac{\sigma_i}{\sigma_i^2+\alpha}\left[(u^T_i H) + (u^T_i\epsilon) \right]v_i = \sum_{i = 1}^r\left[\frac{\sigma_i^2}{\sigma_i^2+\alpha}\left(v^T_i h^*\right) + \frac{\sigma_i}{\sigma_i^2+\alpha}(u^T_i\epsilon) \right]v_i.
\end{equation}
For large singular values $\sigma_i\gg\sqrt{\alpha}$, $\theta_i(\alpha)\approx 1/\sigma_i$ and $\vartheta_i(\alpha) :=\sigma_i\theta_i(\alpha)\approx 1$. These show that there is almost no effect compared to the $h_{LS}$ stable solution for the both the signal and the noise components. However, for small singular values $\sigma_i\ll\sqrt{\alpha}$, $\theta_i(\alpha)\ll 1/\sigma_i$, implying a strong damping of small singular values, strongly suppressing the noise/variance (contributions from these directions are shrunk). 

Taking the expected value from~\eqref{Tik} we have
\[
h_\alpha - E[h_\alpha] = \sum_{i=1}^n\frac{\sigma_i}{\sigma_i^2 + \alpha}(u_i^T\epsilon)v_i.
\]
Next, we show that the total variance decreases with $\alpha$. Since $h_\alpha$ is a random variable as it depends on random noise $\epsilon$, and because $\cov(\epsilon) = E\left[\epsilon\epsilon^T\right] = \sigma_\epsilon^2 I$, we have along direction $v_i$ (in mode $i$) that $v^T_i(h_\alpha - E[h_\alpha]) = \theta_i(\alpha)(u^T_i\epsilon)$. Therefore, the variance in mode $i$ is,
\begin{align*}
\var[v^T_ih_\alpha] &= E\left[\left(v^T_ih_\alpha - E[v^T_ih_\alpha] \right)^2\right] = \theta_i(\alpha)^2E\left[\left(u^T_i\epsilon\right)^2\right] = \theta_i(\alpha)^2\var\left[u^T_i\epsilon\right] \\
&= \theta_i(\alpha)^2u^T_i\cov[\epsilon]u_i = \theta_i(\alpha)^2\sigma_\epsilon^2 = \sigma_\epsilon^2\frac{\sigma_i^2}{(\sigma_i^2+\alpha)^2}.
\end{align*}
This shows how much smaller the variance is in mode $i$ under Tikhonov compared to the LS in that mode. 
For each $i$, define $\chi_i(\alpha): = \sigma_i^2/(\sigma_i^2+\alpha)^2$. Then $\chi_i'(\alpha) = -2\sigma_i^2/(\sigma_i^2+\alpha)^3<0$, showing that every directional variance $\sigma_\epsilon^2\chi_i(\alpha)$ strictly decreases with $\alpha$. Hence, the total variance $\sum_{i=1}^r\var[v^T_ih_\alpha]$ decreases with $\alpha$. Alternatively, since for the unregularized estimator ($\alpha=0$) we have  $\var[v^T_ih_{LS}] = \sigma_\epsilon^2/\sigma_i^2$, it follows that $\var[v^T_ih_\alpha]\le \var[v^T_ih_{LS}]$.

Similarly, for small singular values, $\vartheta_i(\alpha)\ll 1$, showing that the true signal is heavily shrunk. This signal shrinkage is the bias introduced during regularization. Though the split~\eqref{Tik} shows the signal and noise components, in practice you cannot cleanly separate the signal part from the noise part. Bias is the precise measure of signal shrinkage because it's noise-free. In mode $i$ we have from the first equation in~\eqref{std} and~\eqref{Tik},
\begin{align*}
\bias_i(h_\alpha) & = E[v^T_ih_\alpha] - v^T_ih^* = -\left(\frac{\sigma_i^2}{\sigma_i^2+\alpha}-1\right)(v^T_ih^*).
\end{align*}
Bias quantifies exactly how much of the signal in mode $i$ is systematically lost due to regularization. The total bias is
\[
\bias(h_\alpha) = E[h_\alpha] - h^* = -\sum_{i=1}^r\left(\frac{\sigma_i^2}{\sigma_i^2+\alpha}-1\right)(v^T_ih^*)v_i.
\]
Observe that the bias starts at $0$ when $\alpha = 0$ and increases as $\alpha$ increases because each factor $-\alpha/(\sigma_i^2+\alpha)$ becomes more negative in magnitude (approaches $-1$ as $\alpha\to\infty$). In other words, the gap between $E[h_\alpha]$ and $h^*$ increases.

In conclusion, Tikhonov targets the ill-posed directions (modes with small $\sigma_i$) and significantly reduces the noise there, strongly suppressing the variance, while leaving well-posed directions almost untouched. But it also shrinks the true signal there (bias), to prevent noise from exploding there. 

Now, if $\kappa(A) = \sigma_{max}/\sigma_{min} = 1$, then all nonzero singular values of $A$ satisfy $\sigma_i = \sigma>0$ for $i=1,\cdots, r$. The matrix $A$ is perfectly conditioned. There is no instability. In this case the unregularized LS solution is the best estimator. When $\alpha I$ is added as in~\eqref{TikhonovSol}, we are biasing the solution away from the true LS solution. In this case, $h_{\alpha} = \frac{\sigma^2}{\sigma^2+\alpha}h_{LS}$, and the shrinkage factor $\sigma^2/(\sigma^2+\alpha)$ is uniform, that is, every mode is shrunk the same. It adds bias everywhere equally. However, because there are no small singular values to cause variance blow-up, the unregularized variance $\sigma_\epsilon^2/\sigma^2$ is already modest and balanced. Reducing it further to $\sigma_\epsilon^2\left(\frac{\sigma}{\sigma^2+\alpha}\right)^2$ yields only marginal benefit. Although variance is reduced, it is already modest in this setting and therefore never dominates the total risk. Thus, the bias term $\left(1-\frac{\sigma^2}{\sigma^2+\alpha}\right)h^*$ is the main effect and its norm
\[
\|\bias(h_\alpha)\|^2 = \left(1-\frac{\sigma^2}{\sigma^2+\alpha}\right)^2\|h^*\|^2
\]
grows quadratically with $\alpha$, regardless of noise, while the variance term $\var(h_\alpha) = r\sigma_\epsilon^2\left(\frac{\sigma}{\sigma^2+\alpha}\right)^2$ decreases with $\alpha$, but wasn't a problem to begin with (it starts from a reasonable value already). Consequently, the error
\[
E\left[\|h_\alpha - h^*\|^2\right] = \|\bias(h_\alpha)\|^2 + \var(h_\alpha)
\]
increases with $\alpha$. Hence regularization is counterproductive. 

This result highlights that Tikhonov regularization is designed to counteract
instabilities associated with small singular values. When such instabilities
are absent, as in the perfectly conditioned case $\kappa(A)=1$, regularization
introduces bias without a compensating variance reduction and therefore degrades
the reconstruction.
\end{proof}

Although identity Tikhonov regularization may reduce the variance
of the noisy reconstruction, it does so by shrinking all singular modes
uniformly and therefore has no mechanism for selectively penalizing
rough or noise-dominated components. This motivates the use of a
nonidentity regularization operator $L$, considered next.

\subsection{Denoising with Tikhonov regularization and GSVD for well-conditioned operators}
In this section, we show that if $L\ne I$ and $\kappa(A) = 1$, Tikhonov regularization combined with Generalized Singular Value Decomposition (GSVD) can be a great tool to denoise the noisy signal. The result is given below.
\begin{proposition}\label{prop:Tikhonov_LnoteqI}
Let $A\in\mathbb{R}^{m\times n}$ and $L\in\mathbb{R}^{q\times n}$, and assume that the stacked matrix $[A\; L]$ has full column rank $n$, so that the Tikhonov problem~\eqref{TikhonovReg} admits a unique minimizer for every $\alpha>0$. 
Then, even when $\kappa(A)=1$, Tikhonov regularization with $L\neq I$ reduces variance in the reconstruction by selectively damping components associated with the penalty operator $L$, while leaving unpenalized modes unaffected.
\end{proposition}

\begin{proof}
By the GSVD of the pair $(A,L)$, there exist orthogonal matrices $U\in\mathbb{R}^{m\times m}$, $V\in\mathbb{R}^{q\times q}$, an invertible $Z\in\mathbb{R}^{n\times n}$, and diagonal nonnegative matrices $C = \diag(c_1,\cdots,c_n)$, $S = \diag(s_1,\cdots,s_n)$ with $c_i^2 + s_i^2 = 1$ such that
\begin{equation}\label{GSVD}
A = UCZ^{-1},\quad L = VSZ^{-1}.
\end{equation}
The triples $(c_i, s_i,z_i)$, where $z_i$ are columns of $Z$, are the GSVD modes. The generalized singular values are $\gamma_i = c_i/s_i$ for those indices with $s_i\ne 0$. The condition number in GSVD is defined as $\max_i\gamma_i/\min_i\gamma_i$.

Now, let $h = Zy$. Using~\eqref{GSVD} and orthogonality we have
\begin{align*}
J(y) &= \|Ah - H^\epsilon\|^2_2 + \alpha\|Lh\|^2_2 = \|UCy-H^\epsilon\|^2_2 + \alpha\|VSy\|^2_2 \\
& = \|U^T(UCy-H^\epsilon)\|^2_2 + \alpha\|V^T(VSy)\|^2_2 = \|Cy-U^TH^\epsilon)\|^2_2 + \alpha\|Sy\|^2_2\\
& = \sum_{i = 1}^n\left[ (c_iy_i-\beta_i)^2 + \alpha (s_iy_i)^2\right],\quad\text{where $\beta_i = u^T_iH^\epsilon$}.
\end{align*}
For each $i$, we minimize $j_i(y_i) = (c_iy_i-\beta_i)^2 + \alpha (s_iy_i)^2$ with respect to $y_i$ to get
\[
y_i = \frac{c_i}{c_i^2+\alpha s_i^2}\beta_i\quad\text{and so}\; y = \sum_{i = 1}^n\psi_i(\alpha)\beta_ie_i\;\text{and $e_i$ is the standard $i$-th basis vector, where}
\]
\[
\psi_i(\alpha) = \frac{c_i}{c_i^2 + \alpha s_i^2}.\quad\text{Let also\; $\omega_i(\alpha) = \frac{c_i^2}{c_i^2 + \alpha s_i^2}$}.
\]
Let $\{\tilde{z}_i\}$ defined by $\tilde{z}_i:=(Z^{-1})^Te_i$ be the dual (bi-orthogonal) basis which satisfies $\tilde{z}_i^Tz_j = \delta_{ij}$. 
Now, observe that from the first equation in~\eqref{GSVD} and the dual directions we have for each $i$, 
\[
u_i^TAh^* = (Ue_i)^TAh^* = e_i^TCZ^{-1}h^* = (C^Te_i)^TZ^{-1}h^* = c_i\tilde{z}_i^Th^*.
\]
Since $Ze_i = z_i$ it follows that
\begin{equation}\label{h_alpha}
h_\alpha = Zy = \sum_{i = 1}^n\psi_i(\alpha)\left(u_i^TAh^* + u_i^T\epsilon\right)z_i = \sum_{i=1}^n\left[ \omega_i(\alpha)(\tilde{z}_i^Th^*) + \frac{1}{c_i}\omega_i(\alpha)(u_i^T\epsilon)\right]z_i.
\end{equation}
Notice that in the generalized case $L\ne I$, the filters $\psi_i(\alpha)$ depend on $s_i$. So, the shrinkage is accomplished according to the prior/penalty, not according to $A'$s conditioning. 

Applying the dual directions in~\eqref{h_alpha} we have for each $i$,
\[
\tilde{z}_i^Th_\alpha = \omega_i(\alpha)\tilde{z}_i^Th^* + \psi_i(\alpha)u_i^T\epsilon.
\] 
Taking the expected value and recalling that $E[\epsilon] = 0$ yield
\begin{align*}
E[\tilde{z}_i^Th_\alpha] & = \omega_i(\alpha)(\tilde{z}_i^Th^*),
\end{align*}
from which we deduce that the bias along the dual directions is
\begin{equation}\label{bias_GSVD}
\bias(h_\alpha) = E[\tilde{z}_i^Th_\alpha] - \tilde{z}_i^Th^* = (\omega_i(\alpha)-1)\tilde{z}_i^Th^*.
\end{equation}
Thus, modes with large $s_i$ (those emphasized by the penalty) are strongly shrunk, hence more penalized, while modes with $s_i=0$ (those in $\ker L$) are unbiased.

For the directional variance along the dual directions we have
\begin{align}\label{var_GSVD}
\var[\tilde{z}_i^Th_\alpha] & = E\left[\left( \tilde{z}_i^Th_\alpha -E[\tilde{z}_i^Th_\alpha]\right)^2\right] = E\left[\left(\psi_i(\alpha)u_i^T\epsilon \right)^2 \right] = \psi_i(\alpha)^2\var[u_i^T\epsilon]\notag\\
& = \psi_i(\alpha)^2u_i^T\cov(\epsilon)u_i =\psi_i(\alpha)^2u_i^T(\sigma_\epsilon^2I)u_i = \psi_i(\alpha)^2\sigma_\epsilon^2 = \var[\tilde{z}_i^Th_{LS}]\omega_i(\alpha)^2.
\end{align}
Since
\[
\frac{d}{d\alpha}\omega_i(\alpha)^2 = -\frac{2c_i^2s_i^2}{(c_i^2 + \alpha s_i^2)^3} <0 \;\text{for all $\alpha>0$},
\]
then every directional variance $\var[\tilde{z}_i^Th_\alpha]$ decreases with $\alpha$. The total variance (expected squared deviations)
\begin{align*}
E\left[\|h_\alpha -E[h_\alpha]\|^2 \right] &= \mathrm{tr}(\cov(h_\alpha)) = \tr E\left[(h_\alpha - E[h_\alpha]) (h_\alpha - E[h_\alpha])^T\right] = \sum_{i=1}\psi_i(\alpha)^2\sigma_{\epsilon}^2\|z_i\|^2,
\end{align*}
also decreases with $\alpha$.

Now, if $\kappa(A)=1$, not all $c_i$ must be equal. The GSVD balances $A$ and $L$, and the $c_i$ values depend on the interaction between the two operators, not just on $A$'s conditioning. The stabilizing/selection action comes from $s_i$. More specifically, from ~\eqref{h_alpha},~\eqref{bias_GSVD}, and~\eqref{var_GSVD} we have the following cases for the GSVD regime.

\noindent{\bf Case 1:} $s_i = 0$ and $c_i=1$. So $\bias(h_\alpha) = 0$ which implies no signal shrinkage. Moreover, $\var[\tilde{z}_i^Th_\alpha] = \var[\tilde{z}_i^Th_{LS}]$, implying that there is no attenuation of noise. These are the ``unpenalized modes" which are left untouched by regularization.

\noindent{\bf Case 2:} $c_i\to 0$ and $s_i=1$. The signal component in mode $i$ is completely suppressed. $\var[\tilde{z}_i^Th_\alpha] = 0$, so noise in this mode is fully suppressed. Consequently, the solution $h_\alpha$ contains no contributions along these ``fully penalized modes.''

\noindent{\bf Case 3:} $0<c_i<1$ and $s_i=\sqrt{1-c_i^2}>0$ or $0<s_i<1$ and $c_i=\sqrt{1-s_i^2}>0$. This case is the main GSVD regime where Tikhonov regularization balances signal and noise. This is where the trade-off really shows.

For small $\alpha$:
\[
\omega_i(\alpha) = 1-\alpha\frac{s_i^2}{c_i^2} + O(\alpha^2)\approx 1- \alpha\frac{s_i^2}{c_i^2}.
\]
\begin{itemize}
\item If $\alpha s_i^2/c_i^2 \approx 1$, then $\omega_i(\alpha)\ll 1$. Bias increases significantly, which implies strong signal shrinkage in mode $i$. Moreover, $\var[\tilde{z}_i^Th_\alpha] = \var[\tilde{z}_i^Th_{LS}]\omega_i(\alpha)^2$. Thus, variance decreases considerably, implying huge noise reduction in mode $i$. Here $h_{LS}$ denotes the unregularized least-squares solution, corresponding to $\alpha = 0$.
\item If $\alpha s_i^2/c_i^2 < 1$ but not close to 1, then $\omega_i(\alpha)<1$. Bias moderately increases, which implies a slight signal shrinkage in mode $i$. At the same time, variance slightly decreases, showing moderate noise reduction in mode $i$. 
\end{itemize}
\hspace{0.2in} For large $\alpha$: $\omega_i(\alpha)<1$. Thus, signal and noise are both progressively damped in proportion to $\alpha$ in mode $i$. Moreover, $\omega_i(\alpha)\to 0$ as $\alpha\to\infty$, implying full suppression in mode $i$.\\ 

In contrast to the case $L=I$ considered in Proposition~\ref{prop:Tikhonov_with_L=I}, 
the GSVD framework shows that when $L\neq I$, regularization acts along directions determined by the penalty operator rather than by the conditioning of $A$. 
As a result, even if $A$ is well conditioned, Tikhonov regularization can effectively reduce variance by suppressing components associated with large $s_i$, while leaving unpenalized modes intact. 
This selective smoothing mechanism explains the denoising behavior observed in the numerical experiments.
\end{proof}

\section{Numerical results}\label{numerical_results}
In this section, we present the results of our numerical experiments. For simplicity, we assume that $\Omega$ is a domain in $\mathbb{R}^2$, namely, the square $\Omega = [-1,1]\times [-1,1]$. We denote $x\in\Omega$ by $x = (x_1,x_2)$. The final time is taken to be $T = 1$. We use MATLAB to perform our numerical experiments. 
The numerical approximations employ a finite element discretization in space and a backward Euler difference scheme in time to approximate solutions of systems~\eqref{parab1}--\eqref{bcparab1} and~\eqref{inv1}--\eqref{integparab1}, and the right-hand sides of Equations~\eqref{sol_h} and ~\eqref{sol_u}. The reader interested in learning about the Finite Element Method for parabolic problems is encouraged to consult the references~\cite{lewis2004fundamentals, Li-Chen, Thomee}. The PDE toolbox in MATLAB and its documentation on the Finite Element Method would also be an extremely valuable resource. For the remainder of this section, $N$ will denote the number of subintervals in $t$ and $\tau = \tau_{x_1} = \tau_{x_2}$ will denote the mesh size in both $x_1$ and $x_2$ directions. We'll use $\tau = 0.05$ unless otherwise specified.

To check the accuracy of the approximations, we compute the error using the parabolic H\"{o}lder full norm as defined by~\eqref{full_norm}. For a given state function $u$ and its reconstruction $\bar{u}$, let $e(t,x_1,x_2): = u(t,x_1,x_2) - \bar{u}(t,x_1,x_2)$. For the state function $u$, we consider the relative error of the final-time spatial profile (slice $t = T$)
\begin{equation}\label{final_time_norm}
\frac{|e(T,\cdot)|_{1+\delta/2;2+\delta;Q_T}}{|u(T,\cdot)|_{1+\delta/2, 2+\delta; Q_T}},
\end{equation}
where
\[
|e(T,\cdot)|_{1+\delta/2;2+\delta;Q_T} = |e(T,\cdot)|_{0;\Omega} + \sum_{i=1}^d|\partial_{x_i}e(T,\cdot)|_{0;\Omega} + \lbrack e(T,\cdot)\rbrack_{\delta;\Omega}  
\]
and the relative error of the time profile at a fixed point $(x_0,y_0)$
\begin{equation}\label{fixedPoint_norm}
\frac{|e(\cdot,(x_0,y_0))|_{1+\delta/2;2+\delta;Q_T}}{|u(\cdot,(x_0,y_0))|_{1+\delta/2, 2+\delta; Q_T}},
\end{equation}
where
\[
|e(\cdot,(x_0,y_0))|_{1+\delta/2;2+\delta;Q_T} = |e(\cdot,(x_0,y_0))|_{0;(0,T)} + |\partial_te(\cdot,(x_0,y_0))|_{0;(0,T)} + \lbrack\partial_t e(\cdot,(x_0,y_0))\rbrack_{\delta/2;(0,T)}.  
\]

For the source function $h$ and its reconstruction $\bar{h}$, we compute the relative error 
\begin{equation}\label{source_error}
\frac{|h - \bar{h}|_{1+\delta/2;(0,T)}}{|h|_{1+\delta/2;(0,T)}},
\end{equation}
where
\[
|h - h_\alpha|_{1+\delta/2;(0,T)} = |h-h_\alpha|_{0;(0,T)} + |h' - h'_\alpha|_{0;(0,T)} + [h'-h'_\alpha]_{\delta/2;(0,T)}.
\]
The approximate solutions with no noise and no regularization will be denoted as $\tilde{u}$ and $\tilde{h}$. We will denote the noisy approximations with no regularization as $u_0$ and $h_0$. Finally, the regularized solutions will be denoted as $u_\alpha$ and $h_\alpha$.

It is worth noting that compared to~\eqref{full_norm}, in~\eqref{final_time_norm} the time derivatives pieces are dropped as they are irrelevant on a fixed-time slice. Also, since second partial derivatives are not reliable for $P_1$ FEM, they are dropped. Instead, we use the stable spatial smoothness proxy $\lbrack e(T,\cdot)\rbrack_{\delta;\Omega}$. It gives a robust measure of spatial regularity without needing second derivatives and it's controlled by the full norm. On the other side,~\eqref{fixedPoint_norm} follows because the error is a function of time only.

\subsection{A problem with analytic solutions}\label{analytic_sol}
We have constructed the following problem that cannot be found elsewhere in the literature. Problems with analytic solutions are very useful for inverse problems.  Analytic solutions are important for both understanding the ill-posedness of the problem (the dynamic of the system under noisy measurement in the absence of ill-posedness) and serve as another approach of either simulating the data or computing a solution. 
In order to design the system in question, let
\[
w(x_1,x_2) = \cosh\left(\frac{x_1}{\sqrt{2}}\right) \cosh\left(\frac{x_2}{\sqrt{2}}\right) 
\]
chosen so that $\partial_tw = 0$ and $Lw = -\Delta w + w = 0$. Moreover,
\begin{align*}
\frac{\partial w(t,x_1,-1)}{\partial\nu} &= \frac{1}{\sqrt{2}}\sinh\left(\frac{1}{\sqrt{2}}\right)\cosh\left(\frac{x_1}{\sqrt{2}}\right), \, -1\le x_1\leq 1, \,t\in [0,T],\\
\frac{\partial w(t,1,x_2)}{\partial\nu} &= \frac{1}{\sqrt{2}}\sinh\left(\frac{1}{\sqrt{2}}\right)\cosh\left(\frac{x_2}{\sqrt{2}}\right), \, -1\le x_2\leq 1, \,t\in [0,T],\\
\frac{\partial w(t,x_1,1)}{\partial\nu} &= \frac{1}{\sqrt{2}}\sinh\left(\frac{1}{\sqrt{2}}\right)\cosh\left(\frac{x_1}{\sqrt{2}}\right), \, -1\le x_1\leq 1, \,t\in [0,T],\\
\frac{\partial w(t,-1,x_2)}{\partial\nu} &= \frac{1}{\sqrt{2}}\sinh\left(\frac{1}{\sqrt{2}}\right)\cosh\left(\frac{x_2}{\sqrt{2}}\right), \, -1\le x_2\leq 1, \,t\in [0,T],\\
\end{align*}
and
\[
m:=\int_{[-1,1]^2} w(x_1,x_2)\,dx_1\,dx_2 = 8\sinh^2\left(\frac{1}{\sqrt{2}}\right).
\]
Next, define
\[
\phi_{11}(x_1,x_2) = \cos\left(\frac{\pi(x_1+1)}{2}\right) \cos\left(\frac{\pi(x_2+1)}{2}\right).
\]
It follows that  
\[
\frac{\partial \phi_{11}(x_1,x_2)}{\partial\nu} = 0 \;\text{on $\partial [-1,1]$},\quad \int_{[-1,1]^2} \phi_{11}(x_1,x_2)\,dx_1\,dx_2 = 0.
\]
The problem is to recover the source function $h(t)$ and the state function $u(t,x_1, x_2)$ in the system
\begin{align*}
\partial_tu(t,x) - \Delta u(t,x) + u(t,x)
& = h'(t)w(x_1,x_2) + \phi(t,x_1,x_2), \, (t,x=(x_1,x_2))\in Q_T,\\
u(0,x_1,x_2) &= w(x_1,x_2) + \phi_{11}(x_1,x_2), \quad (x_1,x_2)\in \Omega,\\
\frac{\partial u(t,x_1,-1)}{\partial\nu} &= e^{-t}(1+t)\frac{\partial w(t,x_1,-1)}{\partial\nu} + h(t)\frac{\partial w(t,x_1,-1)}{\partial\nu}, \, -1\le x_1\leq 1, \,t\in [0,T],\\ 
\frac{\partial u(t,1,x_2)}{\partial\nu} &= e^{-t}(1+t)\frac{\partial w(t,1,x_2)}{\partial\nu} + h(t)\frac{\partial w(t,1,x_2)}{\partial\nu}, \, -1\le x_2\leq 1, \,t\in [0,T],\\
\frac{\partial u(t,x_1,1)}{\partial\nu} &= e^{-t}(1+t)\frac{\partial w(t,x_1,1)}{\partial\nu} + h(t)\frac{\partial w(t,x_1,1)}{\partial\nu}, \, -1\le x_1\leq 1, \,t\in [0,T],\\
\frac{\partial u(t,-1,x_2)}{\partial\nu} &= e^{-t}(1+t)\frac{\partial w(t,-1,x_2)}{\partial\nu} + h(t)\frac{\partial w(t,-1,x_2)}{\partial\nu}, \, -1\le x_2\leq 1, \,t\in [0,T],\\
\end{align*}
where $\phi(t,x_1,x_2) = - te^{-t}w(x_1,x_2)$, and subject to the integral condition
\begin{equation*}
\int_{-1}^1\int_{-1}^1u(t,x_1,x_2)\,dx_2dx_1 = \mu(t)=8\left(1+te^{-t}\right)\sinh^2\left(\frac{1}{\sqrt{2}}\right),\quad t\in [0,T].
\end{equation*}
The exact solutions of the above system are given by
\begin{equation}\label{exact_guys}
u(t,x_1,x_2) = \left[e^{-t}(1+t) + h(t)\right]w(x_1,x_2) + e^{-(1+\pi^2/2)t}\phi_{11}(x_1,x_2) \;\text{and}\; h(t) = 1-e^{-t}.
\end{equation}
To obtain the approximate solutions, we first compute the baseline solution $v$ which solves the system~\eqref{parab1}--\eqref{bcparab1} with boundary condition
\[
\frac{\partial v}{\partial\nu} = g(t,x_1,x_2) = e^{-t}(1+t)\frac{\partial w}{\partial\nu},
\]
and then proceed with the algorithm as described in Section~\ref{comp_algo}.
This boundary data follows from the fact that 
\[
v(t,x_1,x_2) = u(t,x_1,x_2) - h(t)w(x_1,x_2) = e^{-t}(1+t)w(x_1,x_2) + e^{-(1+\pi^2/2)t}\phi_{11}(x_1,x_2).
\]
For the remainder of this section, we will use $N = 60$ and $\tau = 0.05$. Figure~\ref{fig:exact_No_reg_N60} shows the graphs of the analytic and numerical solutions $u(T,x_1,x_2)$ and $\tilde{u}(T,x_1,x_2)$, $-1\le x_1, x_2\le 1$, respectively, at the final time $T$, their profiles $u(t,-1,-1)$ and $\tilde{u}(t,-1,-1)$, $t\in [0,T]$, at a point chosen arbitrarily, since a similar graph could be obtained at a different point, and the graphs of the analytic and approximate sources $h(t)$ and $\tilde{h}(t)$, respectively, in the case of no noise and no regularization, i.e., $p=0$ and $\alpha = 0$.
\begin{figure}[htbp]
\centering
\includegraphics[width = 0.9\textwidth, height = .4\textheight]{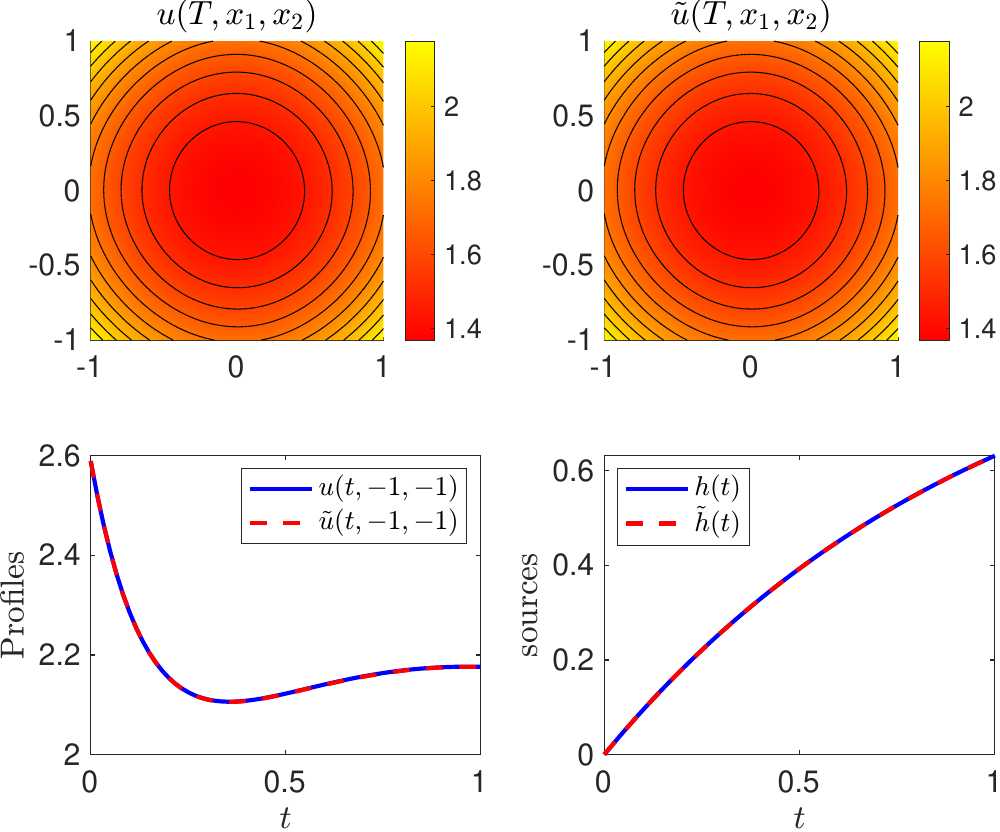}
\caption{Analytic solution $u$ and its approximation $\tilde{u}$ at the final time $T = 1$ (top); profiles of $u$ and its approximation $\tilde{u}$ at $(-1,-1)$ (bottom left); analytic source $h$ and its approximation $\tilde{h}$ (bottom right). $N = 60, \tau = 0.05$ with no noise added and no regularization.}
\label{fig:exact_No_reg_N60}
\end{figure}

The errors resulting from these approximations are given in Table~\ref{table:exact_no_reg_errors}.
\begin{table}[H]
\caption{The relative errors for the solutions in the case of exact data}
\label{table:exact_no_reg_errors}
\centering
\def\arraystretch{1.2}
\begin{tabular}{ccccc} 
\hline                                                                
$p$ & $\alpha$ & $\displaystyle \frac{|e(T,x_1,x_2)|_{1+\delta/2;2+\delta;Q_T}}{|u(T,x_1,x_2)|_{1+\delta/2, 2+\delta; Q_T}}$ & $\displaystyle\frac{|e(t,-1,-1)|_{1+\delta/2;2+\delta;Q_T}}{|u(t,-1,-1)|_{1+\delta/2, 2+\delta; Q_T}}$ & $\displaystyle\frac{|h - \tilde{h}|_{1+\delta/2;(0,T)}}{|h|_{1+\delta/2;(0,T)}}$ \\ [2ex]
\hline                                                          
$0$ & 0 & 0.0001 & 0.0013 & 0.0003 \\
[1ex]
\hline
\end{tabular}
\end{table}

The reconstruction in the case of exact data ($\epsilon = 0$) and no regularization ($\alpha = 0$) seems great; however, in practice, $\epsilon=0$ is impossible, since the data are always corrupted by noise. Even computer-simulated data are affected by roundoff errors. Therefore, this noise-free approximation is not to be trusted. 

Next, we explore the dynamic of the approximations with respect to noise. we first add $p = 1\%, 3\%, 5\%$ noise to the right-hand side of~\eqref{matrixeq} and perform no regularization ($\alpha = 0$). Then use~\eqref{TikhonovSol} to find the approximations. The resulting graphs are given in Figure~\ref{fig:exact_all_Noise_Noreg}.
\begin{figure}[htbp]
\centering
\includegraphics[width = 1\textwidth, height = .35\textheight]{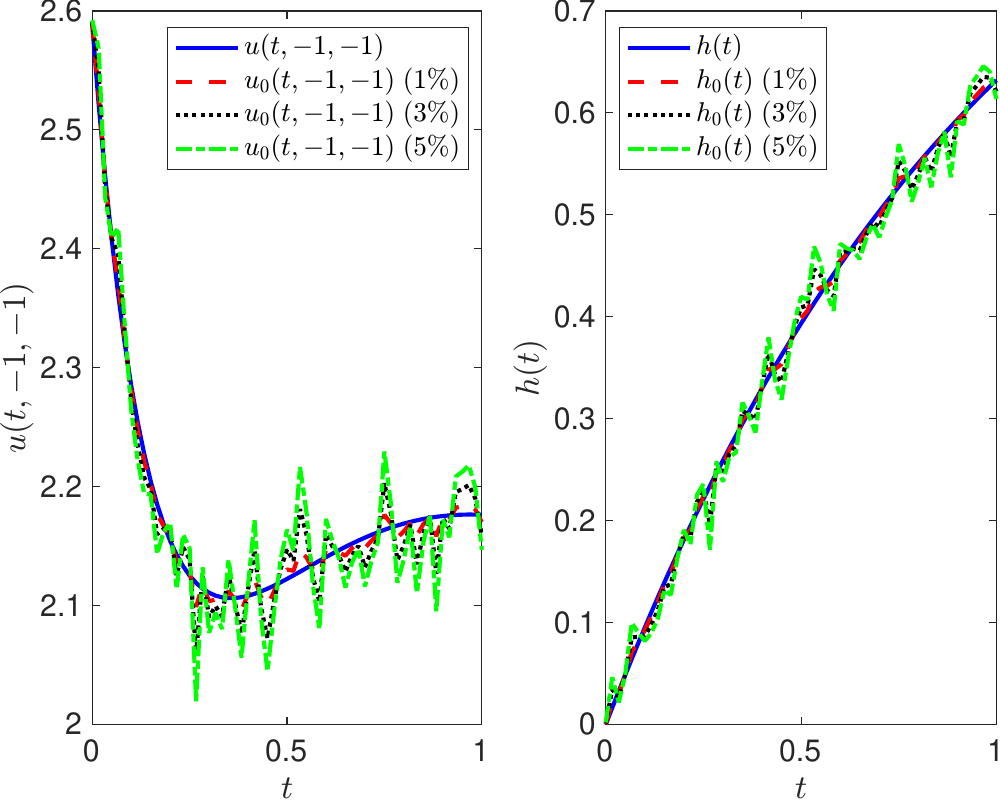}
\caption{Analytic and numerical results of $h(t)$ and $u(t,-1,-1)$ for $p = 1\%, 3\%, 5\%$ and $\alpha = 0$.}
\label{fig:exact_all_Noise_Noreg}
\end{figure}

The resulting errors are given in Table~\ref{table:exact_all_Noise_Noreg}. This shows that the approximations are not stable if no regularization is applied even with such small percentages of noise. 
\begin{table}[H]
\caption{The relative errors of the solutions with $\alpha = 0$ for graphs in Figure~\ref{fig:exact_all_Noise_Noreg}.}
\label{table:exact_all_Noise_Noreg}
\centering
\def\arraystretch{1.2}
\begin{tabular}{cccc} 
\hline                                                                
$p$ & $\alpha$ & $\displaystyle\frac{|e(t,-1,-1)|_{1+\delta/2;2+\delta;Q_T}}{|u(t,-1,-1)|_{1+\delta/2, 2+\delta; Q_T}}$ & $\displaystyle\frac{|h - h_0|_{1+\delta/2;(0,T)}}{|h|_{1+\delta/2;(0,T)}}$ \\ [2ex]
\hline                                                          
$1\%$ & 0 & 0.3556 & 1.2694 \\
[1ex]
$3\%$ & 0 & 1.0668 & 3.8081 \\
[1ex]
$5\%$ & 0 & 1.7781 & 6.3467 \\
\hline
\end{tabular}
\end{table}

Next, we apply Tikhonov regularization~\eqref{TikhonovSol} with $L=I$. This is illustrated in Figure~\ref{fig:exact_all_LeqI_Reg} where we have plotted both the analytical solutions $u(t,-1,-1)$ and $h(t)$ and the regularized solutions  $u_\alpha(t,-1,-1)$ and $h_\alpha(t)$ corresponding to the noisy solutions of Figure~\ref{fig:exact_all_Noise_Noreg}.  
\begin{figure}[H]
\centering
\includegraphics[width = 1\textwidth, height = .3\textheight]{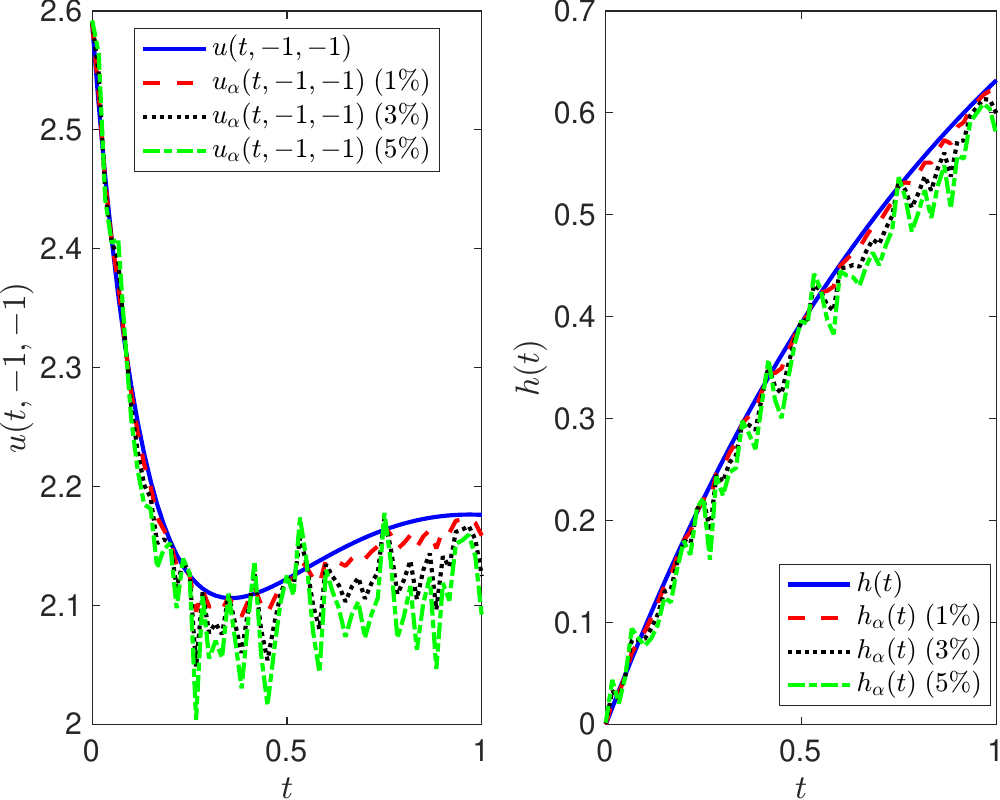}
\caption{Analytic solutions $u(t,-1,-1)$ and $h(t)$ and regularized solutions $u_\alpha(t,-1,-1)$ and $h_\alpha(t)$ corresponding to the noisy solutions of Figure~\ref{fig:exact_all_Noise_Noreg} in the case $L=I$.}
\label{fig:exact_all_LeqI_Reg}
\end{figure}

The relative errors resulting from these reconstructions and the values of the regularization parameter $\alpha$ obtained are given in Table~\ref{table:exact_all_LeqI_Reg_errors}.
\begin{table}[H]
\caption{The relative errors for the solutions and values of the regularization parameter $\alpha$ for graphs in Figure~\ref{fig:exact_all_LeqI_Reg}.}
\label{table:exact_all_LeqI_Reg_errors}
\centering
\def\arraystretch{1.2}
\begin{tabular}{cccc} 
\hline                                                                
$p$ & $\alpha$ & $\displaystyle\frac{|e(t,-1,-1)|_{1+\delta/2;2+\delta;Q_T}}{|u(t,-1,-1)|_{1+\delta/2, 2+\delta; Q_T}}$ & $\displaystyle\frac{|h - h_\alpha|_{1+\delta/2;(0,T)}}{|h|_{1+\delta/2;(0,T)}}$ \\ [2ex]
\hline                                                          
$1\%$ & 0.2551 & 0.3530 & 1.2602 \\
[1ex]
$3\%$ & 0.7846 & 1.0348 & 3.6938 \\
[1ex]
$5\%$ & 1.3410 & 1.6842 & 6.0117 \\
\hline
\end{tabular}
\end{table}

The discrepancy principle curves for obtaining $\alpha$ are provided in Figure~\ref{fig:exact_all_LeqI_MDP}.

\begin{figure}[H]
\centering
\includegraphics[width =1\textwidth, height=0.25\textheight]{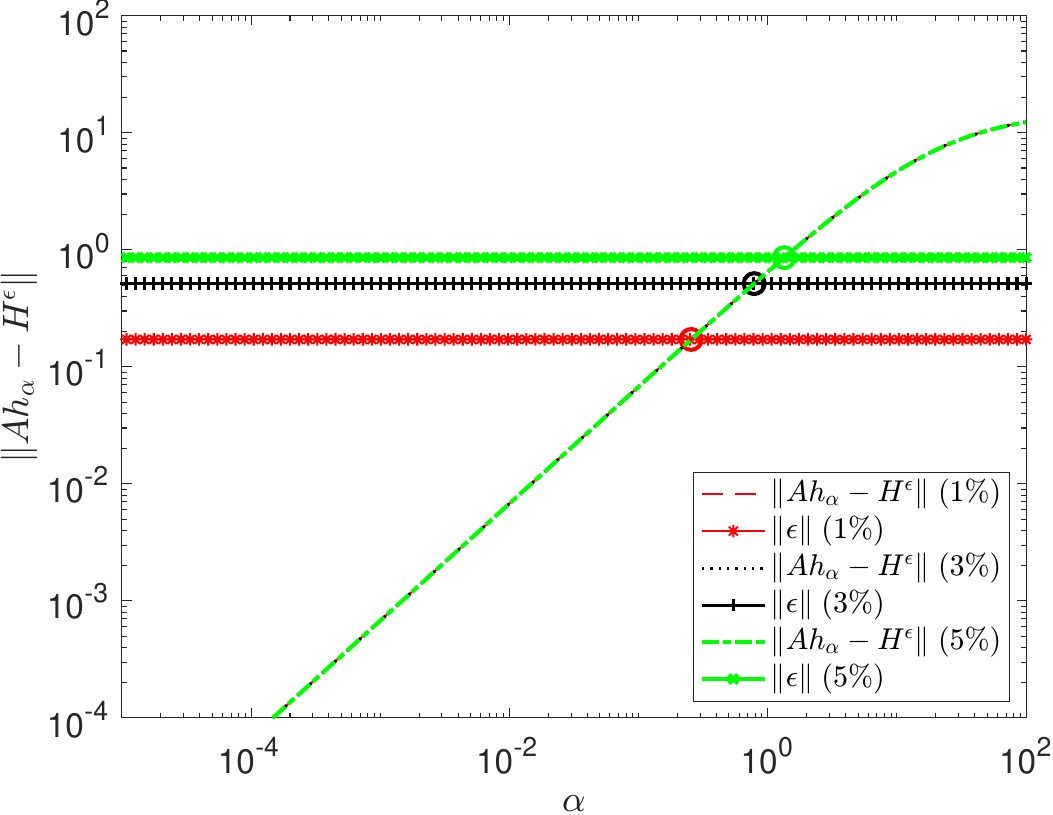}
\caption{The discrepancy principle curves for graphs in Figure~\ref{fig:exact_all_LeqI_Reg} and $\alpha$ in Table~\ref{table:exact_all_LeqI_Reg_errors}.}
\label{fig:exact_all_LeqI_MDP}
\end{figure}

Comparing Figure~\ref{fig:exact_all_Noise_Noreg} with Figure~\ref{fig:exact_all_LeqI_Reg}, it is evident that identity Tikhonov regularization does not effectively denoise the noisy solutions. Moreover, a comparison of Table~\ref{table:exact_all_Noise_Noreg} and Table~\ref{table:exact_all_LeqI_Reg_errors} shows that the relative errors of the regularized solutions remain very close to those of the noisy, unregularized solutions. These observations confirm that, when the coefficient matrix is well conditioned, Tikhonov regularization with $L = I$ provides no significant improvement, in agreement with the theoretical result established in Proposition~\ref{prop:Tikhonov_with_L=I}.

To illustrate our simulations in the case $L\ne I$, the choice of $L$ is the second-derivative regularization operator $L \in \mathbb{R}^{(N-2)\times N}$ defined by the finite-difference approximation
\[
(Lh)_n = \frac{h_{n+2} - 2h_{n+1} + h_n}{\Delta t^2},
\qquad n = 1,\dots,N-2,
\]
that is,
\begin{equation}\label{second_derivative_penalty}
L = \frac{1}{\Delta t^2}
\begin{pmatrix}
1 & -2 & 1 & 0 & \cdots & 0 \\
0 & 1 & -2 & 1 & \ddots & \vdots \\
\vdots & \ddots & \ddots & \ddots & \ddots & 0 \\
0 & \cdots & 0 & 1 & -2 & 1
\end{pmatrix}.
\end{equation}
This matrix better captures the smoothness of the unknown source term $h(t)$, penalizes the discrete curvature of the reconstructed source, and enforces additional smoothness compared to first-derivative regularization. 

This second-derivative penalty is particularly well suited for smooth source functions.
In contrast to first-derivative regularization, it more effectively suppresses spurious oscillations in the derivative of the reconstructed solution. As demonstrated in the numerical experiments, this choice leads to improved accuracy in derivative-sensitive errors in $C^{1+\delta/2, 2+\delta}$ and $C^{1+\delta/2}$-type norms while maintaining stability with respect to noise.

Figure~\ref{fig:exact_all_LnoteqI_Reg} shows the graphs of the exact and regularized solutions corresponding to the noisy solutions of Figure~\ref{fig:exact_all_Noise_Noreg} for $p=1\%, 3\%, 5\%$ obtained with the operator $L$ given by the regularization matrix~\eqref{second_derivative_penalty}. 
\begin{figure}[H]
\centering
\includegraphics[width = 1\textwidth, height = .33\textheight]{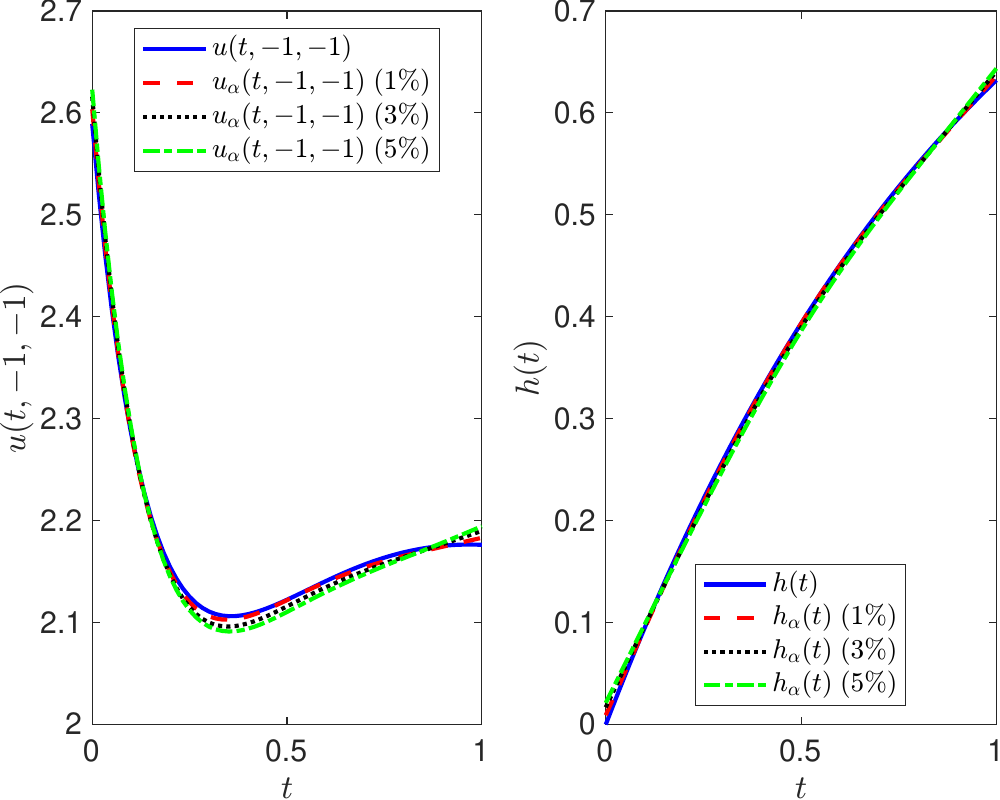}
\caption{The exact and regularized solutions for $p=1\%, 3\%, 5\%$ obtained with the operator $L$ in~\eqref{second_derivative_penalty}.}
\label{fig:exact_all_LnoteqI_Reg}
\end{figure}

The relative errors resulting from the approximations are provided in Table~\ref{table:exact_all_LnoteqI_Reg_errors} together with the values of the regularization parameter $\alpha$ obtained. These errors show an improvement compared to those of noisy solutions in Table~\ref{table:exact_all_Noise_Noreg}.
\begin{table}[H]
\caption{The relative errors for the solutions and values of the regularization parameter $\alpha$ for graphs in Figure~\ref{fig:exact_all_LnoteqI_Reg}.}
\label{table:exact_all_LnoteqI_Reg_errors}
\centering
\def\arraystretch{1.2}
\begin{tabular}{cccc} 
\hline                                                                
$p$ & $\alpha$ & $\displaystyle\frac{|e(t,-1,-1)|_{1+\delta/2;2+\delta;Q_T}}{|u(t,-1,-1)|_{1+\delta/2, 2+\delta; Q_T}}$ & $\displaystyle\frac{|h - h_\alpha|_{1+\delta/2;(0,T)}}{|h|_{1+\delta/2;(0,T)}}$ \\ [2ex]
\hline                                                          
$1\%$ & 0.0025 & 0.0453 & 0.1645 \\
[1ex]
$3\%$ & 0.0096 & 0.0606 & 0.2190 \\
[1ex]
$5\%$ & 0.0163 & 0.0674 & 0.2435 \\
\hline
\end{tabular}
\end{table}

Unlike the not-so-good Tikhonov regularization results of Figure~\ref{fig:exact_all_LeqI_Reg} and Table~\ref{table:exact_all_LeqI_Reg_errors} in the standard case $L=I$, regularization in the general case $L\ne I$ and in particular in this case where the prior/penalty $L$ is given by the second derivative penalty~\eqref{second_derivative_penalty}, performs well, in agreement with the theoretical result established in Proposition~\ref{prop:Tikhonov_LnoteqI}.

As expected, as the noise level increases from $1\%$ to $5\%$, the regularization parameter increases monotonically with the noise level and the relative error in the recovered source increases, while the time-trace error at the observation point increases modestly, indicating stable reconstruction under increasing perturbations and stronger smoothing in the reconstruction.

Although the relative errors in $h$ range from 0.1645 to 0.2435, these are computed in a strong $C^{1+\delta/2}(0,T)$-type norm that includes both $|h'|_{0;(0,T)}$ and the H\"{o}lder seminorm $[h']_{\delta/2; (0,T)}$; these derivative-based terms are well known to amplify small oscillations introduced by noisy data and discrete differentiation~\cite{EnglHankeNeubauer1996,Lunardi1995,Natterer2001}, even when the reconstructed $h$ matches the exact profile closely as in Figure~\ref{fig:exact_all_LnoteqI_Reg}.

The discrepancy principle curves for obtaining $\alpha$ in Table~\ref{table:exact_all_LnoteqI_Reg_errors} are provided in Figure~\ref{fig:exact_all_LnoteqI_MDP}.
\begin{figure}[H]
\centering
\includegraphics[width = 1\textwidth, height = .33\textheight]{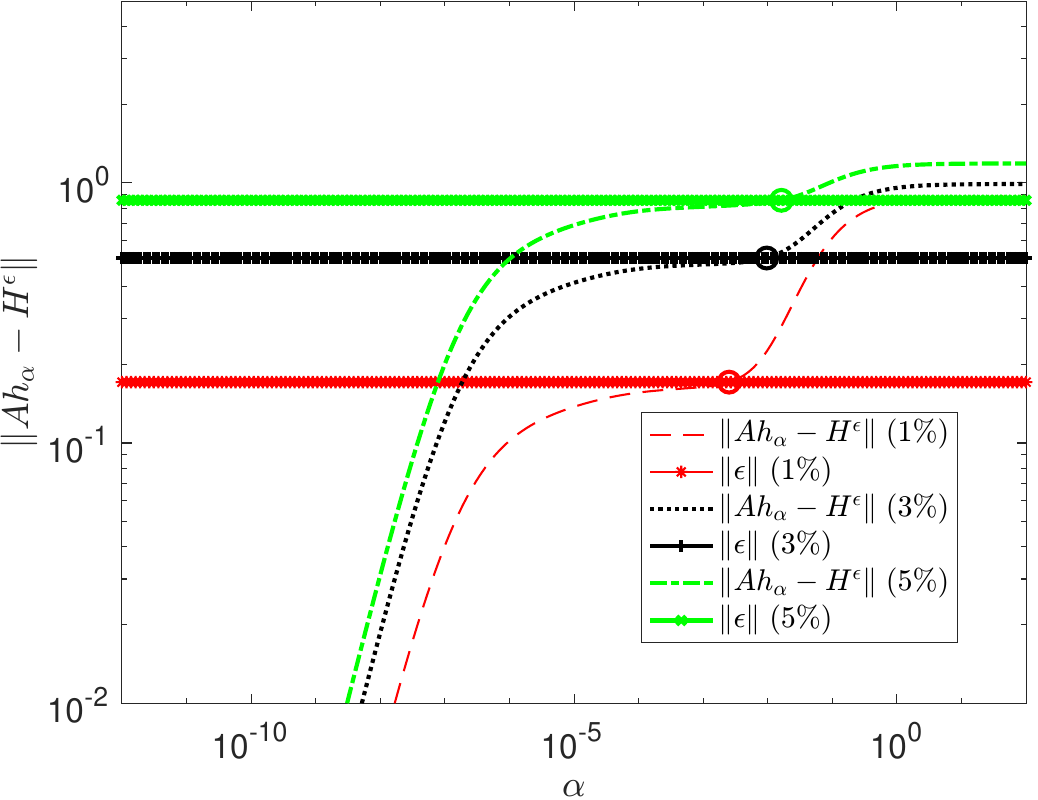}
\caption{The discrepancy principle curves for $p=1\%, 3\%, 5\%$ noise and $\alpha$ in Table~\ref{table:exact_all_LnoteqI_Reg_errors}.}
\label{fig:exact_all_LnoteqI_MDP}
\end{figure}

\subsection{Computational algorithm}\label{comp_algo}
We provide the computational algorithm employed to find approximate solutions using the steps below. The first step is split into two depending on whether we have analytic solutions/real-world data or not. However, Step 2--Step 5 are common to both cases. 
\begin{description}
\item [Step 1a:] {\it If analytic solution or real-world data for which the integral data $\mu(t)$ is available}: approximate the solution $v$ of the forward problem~\eqref{parab1}--\eqref{bcparab1} and calculate its integral over the spatial domain $\Omega$. Then proceed to Step 2--Step 5.
\item [Step 1b:] {\it In the absence of analytic solution or real-world integral data $\mu(t)$}, simulate the data: provide a source function $h(t)$ that is intended to be recovered. Call it the {\it exact source}. With this source given,  solve the problem~\eqref{inv1}--\eqref{bcinv1} for $u$. Call it the {\it exact $u$}. Then compute the integral data $\mu(t)$ given by~\eqref{integparab1}. Next, approximate the solution $v$ of the forward problem~\eqref{parab1}--\eqref{bcparab1} and calculate its integral over the spatial domain $\Omega$. Make sure to use a different mesh size than the one employed to solve the problem~\eqref{inv1}--\eqref{bcinv1} for $u$, in order to avoid committing an inverse crime. Then go to Step 2--Step 5.
\item [Step 2:] Construct the matrix $A$ in~\eqref{matrixeq} and add noise $\epsilon$ to the right-hand side of~\eqref{matrixeq}. Moreover, construct a differential regularization matrix $L$, like the one in~\eqref{second_derivative_penalty}.
\item [Step 3:] Use Tikhonov regularization formula~\eqref{TikhonovSol} to compute the approximate source $h_\alpha$. You will need to use a method to estimate the regularization parameter $\alpha$ such as the discrepancy principle or others.
\item [Step 4:] Find $u_\alpha$, the approximate $u$ from Equation~\eqref{sol_u}, keeping in mind to use $h_\alpha$, not $h$.
\item [Step 5:] Compute the errors in the solutions.
\end{description}

\subsection{A problem with no analytic solutions}\label{piece_wise_approximation}
In section~\ref{analytic_sol} we applied Step (1a) of the algorithm to differentiable functions $h(t)$ and $u(t,x_1,x_2)$ with analytic solutions. In real-world applications, analytic solutions are almost never available. In this section we construct a more severe discontinuous source function as follows
\begin{align}\label{discontinuousSource}
h(t) = 
\begin{cases}
0 &\quad\text{if $0\le t < \frac{T}{5}$} \\
1 &\quad\text{if $\frac{T}{5}\le t < \frac{T}{2}$} \\
0 &\quad\text{if $\frac{T}{2}\le t < \frac{3}{4}T$} \\
1 &\quad\text{if $\frac{3}{4}T\le t \le T$}, \\
\end{cases}
\end{align}
and apply Step (1b) of the algorithm to demonstrate the accuracy of our algorithm. For simplicity, we assume that for all $i, j = 1, \cdots, d$,
\begin{align*}
a^{ij}(t,x_1,x_2) &= a(t,x_1,x_2) = (x_1^2 + x_2^2)t,\quad b^i(t,x_1,x_2) \equiv 0,\quad c(t,x_1,x_2)\equiv 1,\\ 
w(t,x_1,x_2) &\equiv 1,\quad\phi(t,x_1,x_2)\equiv 0,\\
g(t,x_1,x_2) &= A\left[1+\sin\left(2\pi(x_1 - ct)\right)\cos(\pi x_2)\right]\;\text{a traveling-wave flux with $A=1$ and $c=0.5$},\\
u_0(x_1,x_2) &= \frac{1}{4}\left(1+\cos(\pi x_1)\right)\left(1+\cos(\pi x_2)\right).
\end{align*}
So we solve system~\eqref{inv1}--\eqref{bcinv1} with the above data and source term $h'(t) + h(t)$. Since $h(t)$ as defined in~\eqref{discontinuousSource} is piecewise constant with jump discontinuities, its derivative is not an ordinary function, it is a distribution (a sum of Dirac deltas) located at the jump times:
\begin{equation}\label{eq:hprime}
h'(t) = \delta\left(t-\frac{T}{5}\right)
- \delta\left(t-\frac{T}{2}\right)
+ \delta\left(t-\frac{3T}{4}\right),
\end{equation}
where $\delta(\cdot)$ denotes the Dirac delta distribution. In our numerical simulations, the forcing term $h(t)+h'(t)$ is implemented as in~\eqref{discontinuousSource} and~\eqref{eq:hprime} where each delta function is approximated by a narrow Gaussian:
\begin{equation}\label{eq:delta_approx}
\delta(t-t_0)\;\approx\;\frac{1}{\sqrt{2\pi}\,\sigma_{\mathrm{pde}}}\exp\left(-\frac{(t-t_0)^2}{2\sigma_{\mathrm{pde}}^2}\right).
\end{equation}

In the numerical implementation, Dirac delta distributions were approximated by Gaussian mollifiers with standard deviation proportional to the time step. Specifically, we used
$\sigma_{\mathrm{pde}} = 0.5\;\Delta t$ to solve the PDE in order to localize the source within a single time step~\cite{Evans10}.

One thing to avoid when dealing with an inverse problem is the concept of \textit{inverse crime}, which occurs when the same method on the same mesh and time steps is being used to both construct the data and compute the solution~\cite[Chapters 2-3]{Mueller-Siltanen12}. To avoid committing an inverse crime, in the reconstruction process, we take $\tau = 0.06$ (different from $\tau=0.05$ in simulating the data) and keep $N=60$, then interpolate the simulated data into the mesh of the reconstruction. 

The choice of $L$ in this example is the first-derivative regularization operator $L \in \mathbb{R}^{(N-1)\times N}$ defined by the finite-difference approximation on a uniform grid
\[
(Lh)_n = \frac{h_{n+1}-h_n}{\Delta t},
\qquad n = 1,\dots,N-1.
\]
That is,
\begin{equation}\label{diffRegmatrix1}
L = \frac{1}{\Delta t}
\begin{pmatrix}
-1 &  1 &  0 &  0 & \cdots & 0 \\
 0 & -1 &  1 &  0 & \cdots & 0 \\
 0 &  0 & -1 &  1 & \cdots & 0 \\
 \vdots &   &   & \ddots & \ddots & \vdots \\
 0 & \cdots & 0 & 0 & -1 & 1
\end{pmatrix}.
\end{equation}

For the piecewise-constant source, first-derivative regularization is preferable since it promotes near-constant plateaus and suppresses spurious oscillations without enforcing unnecessary curvature smoothness across jump regions, whereas second-derivative regularization tends to round discontinuities and bias the reconstruction near transitions.

In Figure~\ref{fig:piece_NoNoise_NoReg} we have plotted a time trace profile of the exact solution $u$ and its approximation and the exact and approximate source functions with no added noise and no regularization. With no surprise, an eye inspection shows an excellent recovery. 
\begin{figure}[htbp]
\centering
\includegraphics[width = 1\textwidth, height = .35\textheight]{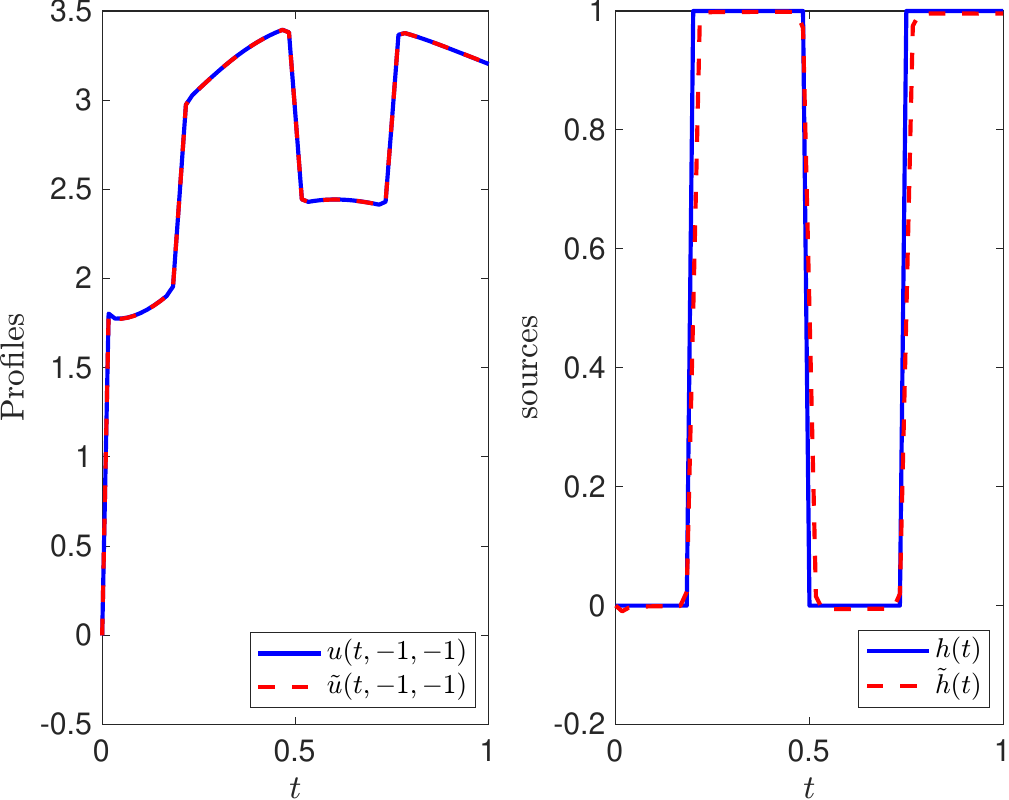}
\caption{Exact and numerical results of $h(t)$ and $u(t,-1,-1)$ obtained with no added noise and no regularization with $N=60$ and $h = 0.05$ for simulating the data and $h = 0.06$ for reconstruction.}
\label{fig:piece_NoNoise_NoReg}
\end{figure}

To compute the relative error in $C^{1+\delta/2,2+\delta}(Q_T)$ and $C^{1+\delta/2}(0,T)$ as defined in~\eqref{final_time_norm}--\eqref{source_error}, we smooth the exact piecewise constant source $h(t)$, its derivative $h'(t)$, and its numerical reconstruction $\tilde{h}(t)$. To this end, observe that the source in~\eqref{discontinuousSource} can be written as
\[
h(t)= H\left(t-\frac{T}{5}\right) - H\left(t-\frac{T}{2}\right) + H\left(t-\frac{3T}{4}\right) - H(t-T),
\]
where $H$ denotes the Heaviside step function. Now, for a fixed scale $\sigma_{\mathrm{err}}>0$, let
\[
\rho_{\sigma_{\mathrm{err}}}(t) = \frac{1}{\sqrt{2\pi}\,\sigma_{\mathrm{err}}} \exp\left(-\frac{t^2}{2\sigma_{\mathrm{err}}^2}\right),
\]
and define the Gaussian--mollified source
\[
h_{\sigma}(t) := (h * \rho_{\sigma})(t) =  \int_0^T h(s)\,\rho_{\sigma_{\mathrm{err}}}(t-s)\,ds.
\]
Using the identity
\[
(H * \rho_{\sigma_{\mathrm{err}}})(t)
=
\frac{1}{2}
\left(
1 + \operatorname{erf}\!\left(\frac{t}{\sqrt{2}\,\sigma_{\mathrm{err}}}\right)
\right),
\]
we obtain the explicit closed form
\begin{align*}
h_{\sigma_{\mathrm{err}}}(t)
&=
\frac{1}{2}
\left(
1 + \operatorname{erf}\!\left(\frac{t-\frac{T}{5}}{\sqrt{2}\,\sigma_{\mathrm{err}}}\right)
\right) -
\frac{1}{2}\left(1 + \operatorname{erf}\!\left(\frac{t-\frac{T}{2}}{\sqrt{2}\,\sigma_{\mathrm{err}}}\right)
\right) +
\frac{1}{2}\left(1 + \operatorname{erf}\!\left(\frac{t-\frac{3T}{4}}{\sqrt{2}\,\sigma_{\mathrm{err}}}\right)
\right) \\
&\quad - \frac{1}{2}\left(1 + \operatorname{erf}\!\left(\frac{t-T}{\sqrt{2}\,\sigma_{\mathrm{err}}}\right)
\right).
\end{align*}
Differentiating explicitly, we obtain
\begin{equation*}
h_{\sigma_{\mathrm{err}}}'(t)
=
\frac{1}{\sqrt{2\pi}\,\sigma_{\mathrm{err}}}
\Bigg[
\exp\!\left(-\frac{(t-\frac{T}{5})^2}{2\sigma_{\mathrm{err}}^2}\right)
-
\exp\!\left(-\frac{(t-\frac{T}{2})^2}{2\sigma_{\mathrm{err}}^2}\right)
+
\exp\!\left(-\frac{(t-\frac{3T}{4})^2}{2\sigma_{\mathrm{err}}^2}\right)
-
\exp\!\left(-\frac{(t-T)^2}{2\sigma_{\mathrm{err}}^2}\right)
\Bigg].
\end{equation*}
For the mollified reconstruction we have
\[
\tilde{h}_{\sigma_{\mathrm{err}}}(t)
:=
(\tilde{h} * \rho_{\sigma_{\mathrm{err}}})(t)
=
\int_0^T \tilde{h}(s)\,\rho_{\sigma_{\mathrm{err}}}(t-s)\,ds.
\]
Since $\tilde{h}$ is available only at discrete times $t_n=n\Delta t$,
the mollification is implemented as a discrete convolution (quadrature)
\[
\tilde{h}_{\sigma_{\mathrm{err}}}(t_n)
\approx
\sum_{m=0}^{N}
\tilde{h}(t_m)\,\rho_{\sigma_{\mathrm{err}}}(t_n-t_m)\,\Delta t,
\qquad n=0,1,\dots,N,
\]
followed by normalization of the discrete kernel so that the weights sum to $1$. We computed the numerical derivative of $\tilde{h}_{\sigma_{\mathrm{err}}}(t)$ with a centered finite difference formula.
Now, $h_{\sigma_{\mathrm{err}}}$, $h'_{\sigma_{\mathrm{err}}}$, and $\tilde{h}_{\sigma_{\mathrm{err}}}$ belong to $C^\infty(0,T)$ and therefore to $C^{1+\delta/2}(0,T)$ for any $\delta\in(0,1)$. 

In our numerical simulations, a slightly larger width $\sigma_{\mathrm{err}} = 2\Delta t$ was employed in the computation of relative errors to obtain a smoother and more robust comparison between the exact and reconstructed sources. See, for example, \cite{Peskin2002} for approximating delta functions numerically and choosing regularization widths proportional to grid size.

It is worth emphasizing that $\sigma_{\mathrm{err}}$ is only used in the calculation of errors when the approximate solutions have already been obtained, while $\sigma_{\mathrm{pde}}$ in~\eqref{eq:delta_approx} is used in the modeling part, to find approximate solutions to the inverse problem.

In Table~\ref{table:piece_NoNoise_NoReg_errors}, we show the relative errors between the exact and approximate solutions.
\begin{table}[H]
\caption{The relative errors between the exact and the approximate solutions with no noise added and no regularization for graphs in Figure~\ref{fig:piece_NoNoise_NoReg}.}
\label{table:piece_NoNoise_NoReg_errors}
\centering
\def\arraystretch{1.2}
\begin{tabular}{cccc}
\hline 
$p$ & $\alpha$ & $\displaystyle \frac{|e(t,-1,-1)|_{1+\delta/2;2+\delta;Q_T}}{|u(t,-1,-1)|_{1+\delta/2, 2+\delta; Q_T}}$ & $\displaystyle \frac{|h - \tilde{h}|_{1+\delta/2;(0,T)}}{|h|_{1+\delta/2;(0,T)}}$ \\ [2ex]
\hline                                                          
$0$ & 0 & 0.0033 & 0.1187 \\
\hline
\end{tabular}
\end{table}

Next, we add $p = 1\%, 3\%, 5\%$ noise to the data and perform no regularization. The results are displayed in Figure~\ref{fig:piece_all_Noise_NoReg}.  
\begin{figure}[htbp]
\centering
\includegraphics[width = 1\textwidth, height = .3\textheight]{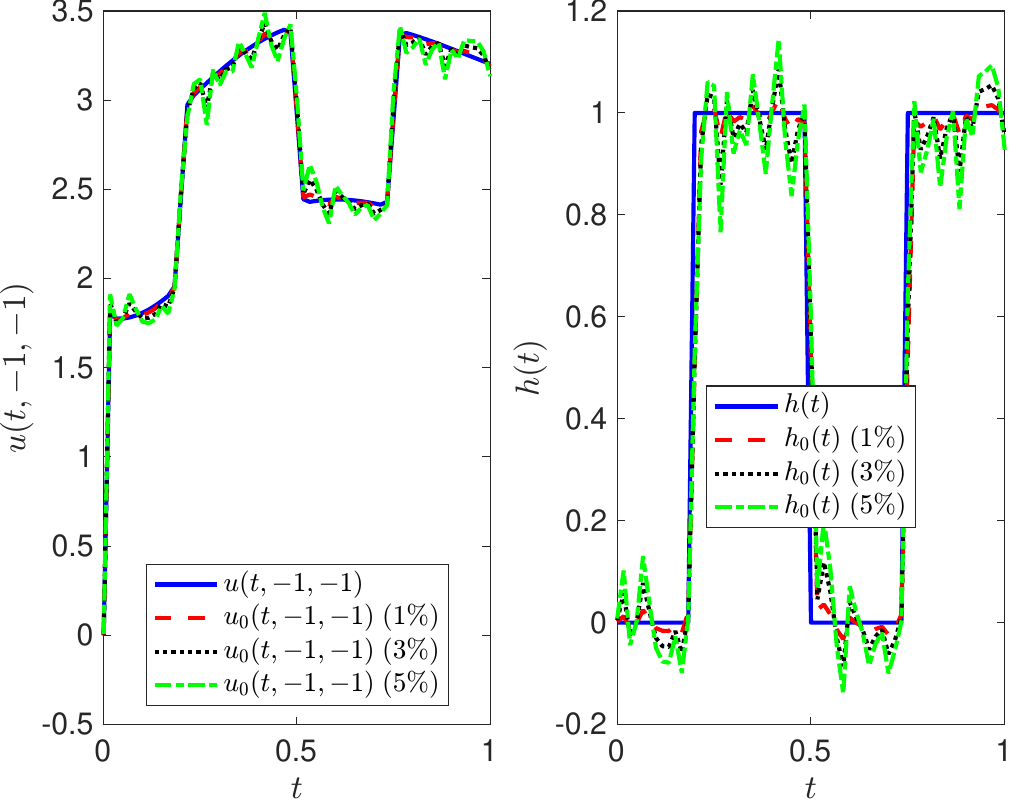}
\caption{Exact and numerical results of $h(t)$ and $u(t,-1,-1)$ obtained with $1\%, 3\%, 5\%$ added noise and no regularization with $N=60$ and $h = 0.05$ for simulating the data and $h = 0.06$ for reconstruction.}
\label{fig:piece_all_Noise_NoReg}
\end{figure}

As one can observe, the results of adding $p = 1\%, 3\%, 5\%$ noise to the data are not catastrophic. This is expected because the operator $A$ in~\eqref{matrixeq} is well-conditioned (a scale of the identity matrix). So, there is no ill-posedness in the problem. Again, the application of regularization in this context is for the purpose of denoising the signal, not to combat instability.

Now, we apply Tikhonov regularization with Morozov's discrepancy principle in estimating the values of the regularization parameter $\alpha$ to denoise the noisy solutions of Figure~\ref{fig:piece_all_Noise_NoReg}. The results are provided in Figure~\ref{fig:piece_all_Noise_Reg}.
\begin{figure}[htbp]
\centering
\includegraphics[width = 1\textwidth, height = .3\textheight]{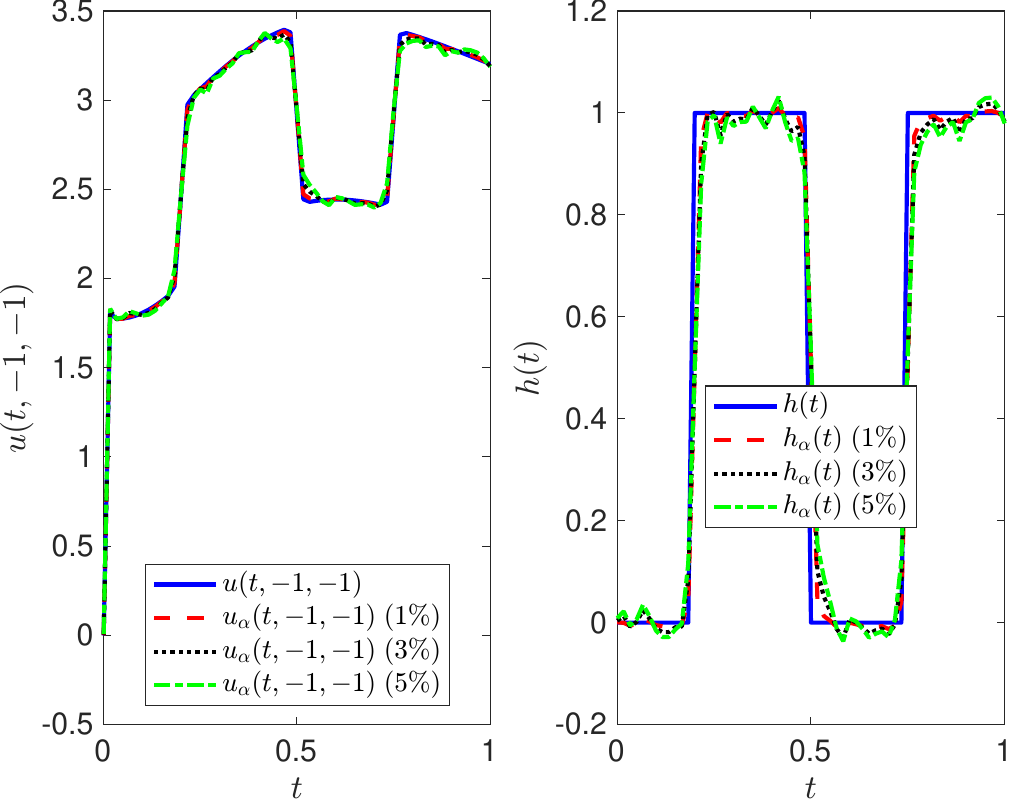}
\caption{The exact and regularized results of $u(t,-1,-1)$ and $h(t)$ for $p\in\{1\%, 3\%, 5\%\}$ corresponding to noisy solutions of Figure~\ref{fig:piece_all_Noise_NoReg}.}
\label{fig:piece_all_Noise_Reg}
\end{figure}

The values of $\alpha$ and the relative errors resulting from the approximations are given in Table~\ref{table:piece_all_Noise_Reg_errors}. The results are consistent with Proposition~\ref{prop:Tikhonov_LnoteqI}

\begin{table}[H]
\caption{The regularization parameters and the relative errors between the exact and the regularized solutions for graphs in Figure~\ref{fig:piece_all_Noise_Reg}.}
\label{table:piece_all_Noise_Reg_errors}
\centering
\def\arraystretch{1.2}
\begin{tabular}{cccc}
\hline 
$p$ & $\alpha$ & $\displaystyle \frac{|e(t,-1,-1)|_{1+\delta/2;2+\delta;Q_T}}{|u(t,-1,-1)|_{1+\delta/2, 2+\delta; Q_T}}$ & $\displaystyle \frac{|h - \tilde{h}|_{1+\delta/2;(0,T)}}{|h|_{1+\delta/2;(0,T)}}$ \\ [2ex]
\hline                                                          
$1\%$ & 2.5009e-04 & 0.0197 & 0.1120 \\ [2ex]
\hline
$3\%$ & 8.7374e-04 & 0.0534 & 0.1499 \\ [1ex]
\hline
$5\%$ & 0.0017 & 0.0809 & 0.1968 \\ [1ex]
\hline
\end{tabular}
\end{table}

The discrepancy principle curves are provided in Figure~\ref{fig:piece_all_MDP}.
\begin{figure}[htbp]
\centering
\includegraphics[width = 1\textwidth, height = .35\textheight]{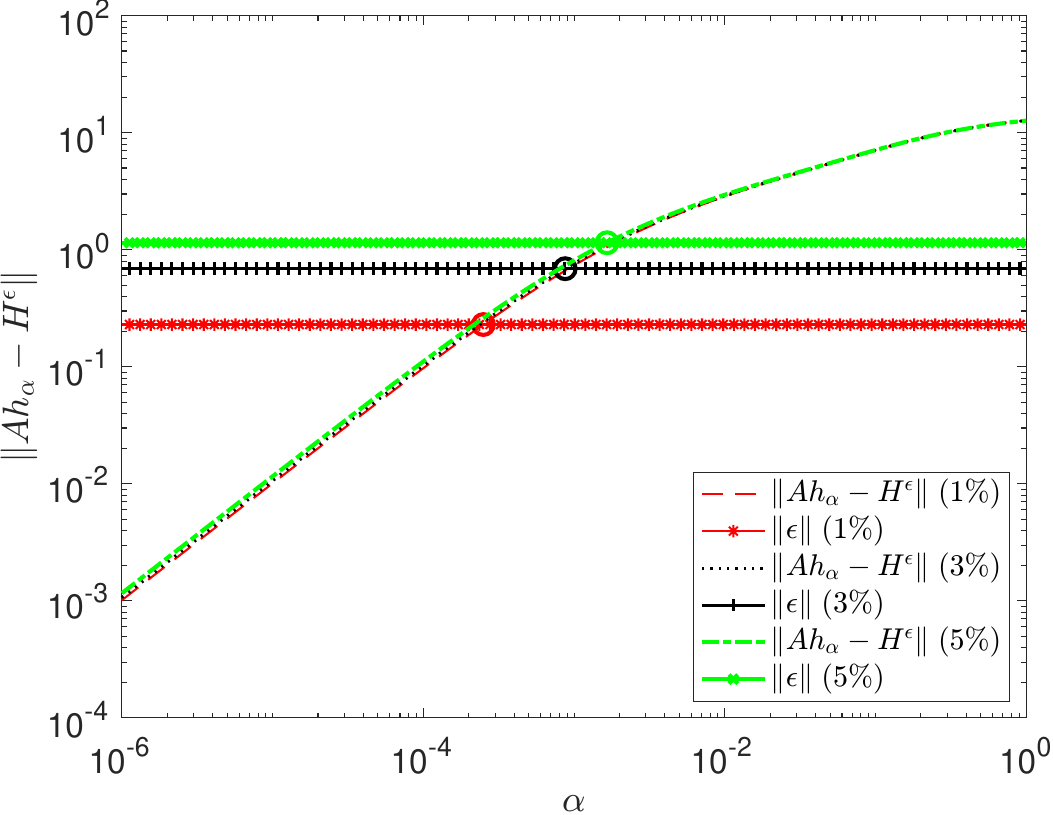}
\caption{The discrepancy principle curves generated with Tikhonov regularization for $p\in\{1\%, 3\%, 5\%\}$ noise levels and regularization parameters $\alpha$ given in Table~\ref{table:piece_all_Noise_Reg_errors}}
\label{fig:piece_all_MDP}
\end{figure}

As the noise level increases from $1\%$ to $5\%$, the discrepancy principle selects progressively larger regularization parameters, resulting in stronger smoothing of the reconstructed source. Although this leads to a moderate increase in the relative source error, the time-trace error at the observation point remains well controlled, demonstrating the stability of the reconstruction with respect to increasing perturbations.

\section{Conclusion and Future Work}\label{conclusion}

In this work, we established the well-posedness and higher H\"older regularity of a
time-dependent inverse source problem for a parabolic equation subject to an integral
constraint and Neumann boundary conditions. The analysis was carried out for a general
parabolic operator with space--time dependent coefficients. In addition, a numerical
algorithm based on a finite element discretization in space and an implicit
time-stepping scheme was developed and implemented.

The resulting discrete inverse problem involves a well-conditioned operator. 
In this setting, we showed theoretically, through the GSVD, that
derivative-based Tikhonov regularization can provide an effective
denoising strategy even though regularization is not required to stabilize
the underlying inverse operator. The regularization parameter was selected using the Morozov discrepancy principle in order to ensure stable reconstructions. Two numerical examples were presented. For the smooth source example, a second-derivative penalty was employed and shown to significantly
improve derivative-sensitive error measures compared to first-derivative
regularization. For the piecewise constant source with jump discontinuities, a
first-derivative penalty was used to suppress oscillations while avoiding unnecessary
curvature smoothing across jump regions. The numerical results demonstrate that the
proposed approach yields accurate and stable reconstructions when errors are measured
in full parabolic H\"{o}lder norms.

Several directions for future work naturally arise from this study. From a numerical
analysis perspective, it would be of interest to derive a rigorous error analysis for
the proposed discretization scheme, including convergence rates with respect to the
mesh size and time step. Another promising direction is the investigation of
alternative regularization strategies, such as total variation or sparsity-promoting
penalties, particularly for source terms with sharp transitions. Finally, it would be
worth exploring whether the integral constraint can be replaced or supplemented by
other types of observational data, such as final-time measurements or partial interior
observations, and to study the corresponding impact on well-posedness and numerical
reconstruction.

\bibliographystyle{plain}
\bibliography{Sample}

@book{Friedman08,
    author  = "Avner Friedman",
    title   = "Partial Differential Equations of Parabolic Type",
    year    = "2008",
    publisher = "Courier Dover Publications",
    volume  = {},
    number  = {},
    pages   = ""
}

@book{Gilbarg-Trudinger96,
    author  = "David Gilbarg and Neil S. Trudinger",
    title   = "Elliptic Partial Differential Equations of second order",
    year    = "1998",
    publisher = "Springer",
    edition = "Revised Third Printing"
}

@book{Krylov96,
    author  = "N. V. Krylov",
    title   = "Lectures on elliptic and parabolic equations in Holder spaces",
    year    = "1996",
    publisher = "American Mathematical Society",
    volume  = {12},
}

@book{Ladyzhenskaia88,
    author  = "Olga Aleksandrovna Ladyzhenskaia and Vsevolod Alekseevich Solonnikov and Nina N. Ural'tseva",
    title   = "Linear and quasi-linear equations of parabolic type",
    year    = "1988",
    publisher = "American Mathematical Society",
    volume  = {23}
}

@article{GlotovHamesMeirNgoma1, 
title={An integral constrained parabolic problem with applications in thermochronology},
author={Glotov, Dmitry and Hames, Willis E and Meir, A. J. and Ngoma, Sedar}, journal={Computers and Mathematics with Applications},
volume={71}, issue={11}, pages={2301-2312}, year={2016}, 
}

@article{GlotovHamesMeirNgoma2, 
title={An inverse diffusion coefficient problem for a parabolic equation with integral constraint},
author={Glotov, Dmitry and Hames, Willis E and Meir, A. J. and Ngoma, Sedar}, journal={International Journal of Numerical Analysis and Modeling},
volume={15}, number={4-5}, pages={552-563}, year={2018}, 
}

@article{HazaneeEtAl13, 
title={An inverse time-dependent source problem for the heat equation},
author={Hazanee, A. and Ismailov, M. I. and Lesnic, D. and Kerimov, N. B.}, journal={Applied Numerical Mathematics},
volume={69}, number={}, pages={13-33}, year={2013}, 
}

@article{Ginder2010, 
title={Construction of solutions to heat-type problems with time-dependent volume constraints},
author={Ginder, Elliott}, journal={Adv. Math. Sci. Appl.},
volume={20}, number={2}, pages={467-482}, year={2010}, 
}

@article{YangEtAl2011, 
title={Inverse problem of time-dependent heat sources numerical reconstruction},
author={Yang, Liu and Dehghan, Mehdi and Yu, Jian-Ning and Luo, Guan-Wei}, journal={Mathematics and Computers in Simulation},
volume={81}, number={}, pages={1656-1672}, year={2011}, 
}

@article{Rundell1980, 
title={Determination of an unknown non-homogeneous term in a linear partial differential equation from overspecified boundary condition},
author={Rundell, William}, journal={Applicable Analysis},
volume={10}, number={}, pages={231-242}, year={1980}, 
}

@article{Kedzierawski1993, 
title={The inverse scattering problem for time-harmonic acoustic waves in an inhomogeneous medium with complex refraction index},
author={Kedzierawski, Andrzej}, journal={Journal of Computational and Applied Mathematics},
volume={47}, number={1}, pages={83-100}, year={1993}, 
}

@article{DamirchiEtAl19, 
title={Numerical investigation of an inverse problem based on regularization method},
author={Damirchi, J. and Yazdanian, A.R. and Shamami, T.R. and Hasanpour, M.}, 
journal={Mathematical Sciences},
volume={13}, number={}, pages={193-199}, year={2019}, 
}

@book{Li-Chen,
  author={Li, Jichun and Chen, Yi-Tung},
  title={Computational Partial Differential Equations using MATLAB},
  year={2009},
  volume={},
  series={Chapman \& Hall/CRC Applied Mathematics and Nonlinear Science Series},
  publisher={CRC PRESS}
}

@book{Thomee,
  author={Thom\'{e}e, Vidar},
  title={Galerkin Finite Element Methods for Parabolic Problems},
  year={2006},
  edition={Second},
  series={Springer Series in Computational Mathematics},
  publisher={Springer}
}

@book{Mueller-Siltanen12,
  author={Mueller, Jennifer L. and Siltanen, Samuli},
  title={Linear and nonlinear inverse problems with practical applications},
  year={2012},
  number={32},
  series={Computational Science \& Engineering},
  publisher={SIAM}
}

@book{Vogel02,
  author={Vogel, Curtis R.},
  title={Computational Methods for Inverse Problems},
  year={2002},
  number={},
  series={Frontiers in Applied Mathematics},
  publisher={SIAM}
}

@book{Evans10,
  author={Evans, Lawrence C.},
  title={Partial Differential Equations},
  year={2010},
  edition={Second},
  volume={19},
  series={Graduate Studies in Mathematics},
  publisher={American Mathematical Society}
}

@book{Isakov2006,
  author={Isakov, Victor},
  title={Inverse Problems for Partial Differential Equations},
  year={2006},
  edition={Second},
  volume={27},
  series={Applied Mathematical Sciences},
  publisher={Springer}
}

@book{lewis2004fundamentals,
  title={Fundamentals of the finite element method for heat and fluid flow},
  author={Lewis, Roland W and Nithiarasu, Perumal and Seetharamu, Kankanhalli N},
  year={2004},
  publisher={John Wiley \& Sons}
}

@inproceedings{morozov1966solution,
  title={On the solution of functional equations by the method of regularization},
  author={Morozov, Vladimir Alekseevich},
  booktitle={Doklady Akademii Nauk},
  volume={167},
  pages={510--512},
  year={1966},
  organization={Russian Academy of Sciences}
}

@book{prilepko2000methods,
  title={Methods for solving inverse problems in mathematical physics},
  author={Prilepko, Aleksey I. and Orlovsky, Dmitry G. and Vasin, Igor A.},
  year={2000},
  publisher={MARCEL DEKKER, INC.},
  series={Monographs and textbooks in pure and applied mathematics}
}

@article{cannon1986diffusion,
  title={Diffusion subject to the specification of mass},
  author={Cannon, John R. and van der Hoek, John},
  journal={Journal of mathematical analysis and applications},
  volume={115},
  number={2},
  pages={517--529},
  year={1986},
  publisher={Elsevier}
}

@article{hasanov2013analysis,
  title={An analysis of inverse source problems with final time measured output data for the heat conduction equation: A semigroup approach},
  author={Hasanov, Alemdar and Slodi{\v{c}}ka, Mari{\'a}n},
  journal={Applied Mathematics Letters},
  volume={26},
  number={2},
  pages={207--214},
  year={2013},
  publisher={Elsevier}
}

@article{hazanee2016reconstruction,
  title={Reconstruction of multiplicative space-and time-dependent sources},
  author={Hazanee, A and Lesnic, D},
  journal={Inverse Problems in Science and Engineering},
  volume={24},
  number={9},
  pages={1528--1549},
  year={2016},
  publisher={Taylor \& Francis}
}

@article{hazanee2013reconstruction,
  title={Reconstruction of an additive space-and time-dependent heat source},
  author={Hazanee, A and Lesnic, D},
  journal={European Journal of Computational Mechanics},
  volume={22},
  number={5-6},
  pages={304--329},
  year={2013},
  publisher={Taylor \& Francis}
}

@book{EnglHankeNeubauer1996,
  author    = {Engl, Heinz W. and Hanke, Martin and Neubauer, Andreas},
  title     = {Regularization of Inverse Problems},
  publisher = {Kluwer Academic Publishers},
  year      = {1996},
  address   = {Dordrecht},
  isbn      = {978-0-7923-4157-4}
}

@book{Natterer2001,
  author    = {Natterer, Frank},
  title     = {The Mathematics of Computerized Tomography},
  publisher = {SIAM},
  year      = {2001},
  address   = {Philadelphia},
  isbn      = {978-0-89871-493-4}
}

@book{Lunardi1995,
  author    = {Lunardi, Alessandra},
  title     = {Analytic Semigroups and Optimal Regularity in Parabolic Problems},
  publisher = {Birkh{\"a}user},
  year      = {1995},
  address   = {Basel},
  isbn      = {978-3-7643-5178-1}
}

@article{Peskin2002,
  author  = {Peskin, Charles S.},
  title   = {The immersed boundary method},
  journal = {Acta Numerica},
  volume  = {11},
  pages   = {479--517},
  year    = {2002},
  doi     = {10.1017/S0962492902000077}
}

@article{ngoma2024well,
  title={Well-posedness and \large{T}ikhonov regularization of an inverse source problem for a parabolic equation with an integral constraint},
  author={Ngoma, Sedar},
  journal={Journal of Inverse and Ill-posed Problems},
  volume={32},
  number={5},
  pages={903--925},
  year={2024},
  publisher={De Gruyter}
}

\end{document}